\documentclass[reqno, 10pt]{amsart}
{
\usepackage{amsmath}
\usepackage[active]{srcltx}
\usepackage{amssymb}
\usepackage[dvips]{graphicx}
\usepackage{xcolor}
\usepackage{hyperref}
\usepackage[margin=1in]{geometry}

\usepackage{amsmath}
\usepackage[active]{srcltx}
\usepackage{amssymb}
\usepackage[dvips]{graphicx}
\usepackage{xcolor}
\usepackage{hyperref}

\numberwithin{equation}{section}
\newtheorem{Theorem}{Theorem}[section]
\newtheorem{Definition}[Theorem]{Definition}
\newtheorem{Lemma}[Theorem]{Lemma}
\newtheorem{Remark}[Theorem]{Remark}

\newtheorem{Corollary}[Theorem]{Corollary}
\setlength{\parskip}{1em}

\usepackage{enumerate}

\begin{document}
\title{ joint normality of representations of numbers: an ergodic approach}

\author[V. Bergelson]{Vitaly Bergelson}
\address[V.Bergelson]{Department of Mathematics \\ Ohio State University \\ Columbus, OH 43210, USA}
\email{vitaly@math.ohio-state.edu}

\author[Y. Son]{Younghwan Son}
\address[Y. Son]{Department of Mathematics \\  POSTECH \\  Pohang, 37673, Korea}
\email{yhson@postech.ac.kr}

\date{\today}
\maketitle

\begin{abstract} 
We introduce an ergodic approach to the study of {\em joint normality} of representations of numbers. For example, we show that for any integer $b \geq 2$ almost every number $x \in [0,1)$ is jointly normal with respect to the $b$-expansion and continued fraction expansion. This fact is a corollary of the following result which deals with {\em pointwise joint ergodicity}:  

Let $T_b:[0,1] \rightarrow [0,1]$ be the times $b$ map defined by $T_b x = bx \,  \bmod \, 1 $ and let $T_G:[0,1] \rightarrow [0,1]$ be the Gauss map defined by $T_G(x) = \{\frac{1}{x}\}$ for $x \ne 0$ and $T_G (0) =0.$  (Here $\{ \cdot \}$ denotes the fractional part.)
For any $f, g \in L^{\infty} (\lambda)$,
\begin{equation*}
\label{abs:eqn}
\lim_{N \rightarrow \infty} \frac{1}{N } \sum_{n=0}^{N-1} f(T_b^{n}x)  \, g(T_G^n x) = \int f \, d \lambda \cdot \int g \, d \mu_G
\quad \text{for almost every } x \in [0,1],
 \end{equation*}
where $\lambda$ is the Lebesgue measure on $[0,1]$ and $\mu_G$ is the Gauss measure on $[0,1]$ given by $\mu_G (A) = \frac{1}{ \log 2} \int_A \frac{1}{1+x} dx$ for any measurable set $A \subset [0,1]$.

We show that the phenomenon of the pointwise joint ergodicity takes place for a wide variety of number-theoretical maps of the interval and derive the corresponding corollaries pertaining to joint normality. 

We also establish the equivalence of various forms of normality and joint normality for representations of numbers, hereby providing a general framework for classical normality results.
\end{abstract} 

\renewcommand{\thefootnote}{\fnsymbol{footnote}} 
    \renewcommand{\thefootnote}{\arabic{footnote}} 

\section{Introduction}
\label{sec:Intro}

For $x \in [0,1)$ and integer $b  \geq 2 $, consider the $b$-expansion of $x$:
\begin{equation}
\label{eq:bexpansion:intro} 
x = \sum_{n=1}^{\infty} \frac{a_n}{b^n},
\end{equation}
where $ a_n \in \{0, 1, \dots, b-1\}$.\footnote{The expansion \eqref{eq:bexpansion:intro} is not unique for rational numbers of the form $\frac{a}{b^n}$. 
To ensure uniqueness, one can assume that $a_n < b-1$ for infinitely many $n$. This point is actually mute since we deal in this paper with results which apply to ``typical" $x \in [0,1).$}
The number $x$ is called {\em simply normal in base $b$} if for any $v \in  \{0, 1, \dots, b-1\}$, 
\[
\lim_{N \rightarrow \infty} \frac{ \# \{1 \leq n \leq N: a_n = v\}}{N} = \frac{1}{b}. 
\]
The number $x$ is called {\em normal in base $b$} if for any $k \in \mathbb{N}$ and for any finite string $S = ( v_1, \dots, v_k ) \in \{0, 1, \dots, b-1\}^k$, we have
\begin{equation}
\label{eq1.2:intro}
\lim_{N \rightarrow \infty} \frac{ \# \{1 \leq n \leq N: a_n = v_1, \dots, a_{n+k-1} = v_k\}}{N} = \frac{1}{b^k}. 
\end{equation}
The classical Borel's normal number theorem (\cite{Bor}) states that almost every $x$ (with respect to the Lebesgue measure $\lambda$) is normal in base $b$ (or $b$-normal). 

Nowadays this result can be easily obtained by means of the ergodic theory.\footnote{Borel derived his theorem from the law of large numbers. One can also derive Borel's theorem by proving first that a number $x \in [0,1)$ is normal in base $b$ if and only if the sequence $(b^n x), n \in \mathbb{N}$, is uniformly distributed  $\bmod \, 1$ and invoking a theorem due to Weyl (see \cite{KN}, Theorem 4.1 on p.32), which states that for any sequence $(a_n)$ of distinct integers, the sequence $(a_n x)$ is uniformly distributed $\bmod \, 1$ for almost every $x$. Various equivalent forms of normality are summarized in Theorem \ref{intro:equiv:b-normal} below. For historical comments on the notion of normal number and the controversies that Borel's paper generated, see Appendix in \cite{BeDoMi}.}  
Indeed, consider the map $T_b:[0,1] \rightarrow [0,1]$ defined by $T_b x = bx \, \bmod 1$. 
Transformation $T_b$ is Lebesgue measure preserving and ergodic, 
 and so by the pointwise ergodic theorem, we have for any $f \in L^1 (\lambda),$
\begin{equation}
\label{b-ergodic}
\lim_{N \rightarrow \infty} \frac{1}{N} \sum_{n=0}^{N-1} f(T_b^n(x)) = \int f \, d \lambda \quad \text{for a.e. } x.
\end{equation}
Let $C_b(S)$ be the set of numbers in $[0,1)$ whose first $k$ terms in $b$-expansion are $(v_1, \dots, v_k)$. 
Applying \eqref{b-ergodic} to the characteristic function  $1_{C_b(S)}$, we get \eqref{eq1.2:intro}.  

Ergodic theory also provides a convenient way of dealing with normality of continued fraction expansions. Let $T_G: [0,1] \rightarrow [0,1]$ be the Gauss map defined by $T_G(x) = \{\frac{1}{x}\}$ for $x \ne 0$ and $T_G (0) =0.$  (Here $\{ \cdot \}$ denotes the fractional part.)
Let $\mu_G$ denote the Gauss measure defined by $\mu_G (A) = \frac{1}{ \log 2} \int_A \frac{1}{1+x} dx$ for any measurable set $A \subset [0,1]$.
It is well-known that $T_G$ is $\mu_G$-preserving and ergodic. (See \cite{Kno} and \cite{Ryll}.) 
By invoking the pointwise ergodic theorem, we have that for any $f \in L^1 (\mu_G),$
\begin{equation}
\label{CFergodic}
\lim_{N \rightarrow \infty} \frac{1}{N} \sum_{n=0}^{N-1} f(T_G^n(x)) = \int f \, d \mu_G \quad \text{for a.e. } x.
\end{equation}

Consider now a regular continued fraction of a number $x \in [0,1)$:
\begin{equation}
\label{CFexpansion}
x = \cfrac{1}{c_1+\cfrac{1}{c_2+\cfrac{1}{c_3 + \cdots}}} = [c_1, c_2, c_3, ...],\footnote{We tacitly assume that $x$ is irrational. The expansions of rational numbers are finite and non-unique, but we may ignore such numbers throughout this paper.} \quad c_n \in \mathbb{N}.
\end{equation}
Given $k \in \mathbb{N}$ and  a finite string $S = (v_1, \dots , v_k) \in \mathbb{N}^k$, let $C_G(S)$ be the set of numbers in $[0,1)$ whose first $k$ terms in the continued fraction representation in \eqref{CFexpansion} are $(v_1, \dots, v_k)$. 
A number $x \in [0,1)$ is called {\em normal with respect to the continued fraction expansion}  (or CF-normal) if for any $k \in \mathbb{N}$ and for any string $S = (v_1, \dots , v_k) \in \mathbb{N}^k,$
\[ \lim_{N \rightarrow \infty} \frac{\# \{1 \leq n \leq N: c_n = v_1, \dots, c_{n+k-1} = v_k\}}{N} = \mu_G (C_G(S)). \]
It is not hard to show that \eqref{CFergodic} implies that almost every $x \in [0,1)$ is CF-normal. 

It follows that almost every $x \in [0,1)$ is simultaneously $b$-normal and CF-normal. 
In this paper we are addressing the following more subtle question: Is it true that almost every $x$ is {\em jointly $b$-normal and CF-normal}? 
By this we mean that for any $k, l \in \mathbb{N}$ and for any $S_1 = (v_1, \dots, v_k) \in \{0, 1, \dots, b-1\}^k$ and $S_2 = (w_1, \dots, w_l) \in \mathbb{N}^l$, 
\begin{equation}
\label{eq:jointnormal1}
   \lim_{N \rightarrow \infty} \frac{1}{N} \left| \{1 \leq n \leq N: a_{n+i-1} = v_i \, (1 \leq i \leq k)  \text{ and } c_{n+j-1} = w_j \, (1 \leq j \leq l) \} \right| 
=  \lambda (C_b(S_1)) \cdot \mu_G (C_G (S_2)). 
\end{equation}
Note that if $x$ is jointly $b$-normal and CF-normal, then it is $b$-normal and CF-normal and moreover
\begin{align} 
&\lim_{N \rightarrow \infty} \frac{ \# \{1 \leq n \leq N: a_{n+i-1} = v_i \, (1 \leq i \leq k)  \text{ and } c_{n+j-1} = w_j \, (1 \leq j \leq l) \}}{N} \\
 &=  \lim_{N \rightarrow \infty} \frac{ \# \{1 \leq n \leq N: a_{n+i-1} = v_i \, (1 \leq i \leq k) \}}{N} 
\times \lim_{N \rightarrow \infty} \frac{\# \{1 \leq n \leq N:  c_{n+j-1} = w_j \, (1 \leq j \leq l)  \}}{N}. 
\end{align}

The answer to this question is YES and is a consequence of the following theorem. 
\begin{Theorem}
\label{thm1-intro}
For any integer $b \geq 2$, for any $f, g \in L^{\infty} (\lambda)$,\footnote{Note that since Lebesgue measure $\lambda$ and Gauss measure $\mu_G$ are mutually absolutely continuous, we have $L^{\infty} (\lambda) = L^{\infty} (\mu_G)$.} 
\begin{equation}
\label{intro:eq:1.4}
\lim_{N \rightarrow \infty} \frac{1}{N } \sum_{n=0}^{N-1} f(T_b^{n}x)  \, g(T_G^n x) = \int f \, d \lambda \cdot \int g \, d \mu_G
 \end{equation}
for almost every $x \in [0,1]$.
\end{Theorem}
Indeed, by Theorem \ref{thm1-intro}, formula \eqref{intro:eq:1.4} is valid for $f = 1_{C_b(S_1)}$ and $g = 1_{C_G(S_2)}$ for any choice of sets $C_b(S_1)$ and $C_G(S_2)$. Since there are only countably many sets of the form $C_b(S_1)$ and $C_G(S_2)$, this gives \eqref{eq:jointnormal1}.

There are numerous examples of concrete constructions of base $b$ normal numbers and CF-normal numbers. (See, for instance, \cite{Ch}, \cite{CoEr}, \cite{DaEr}, \cite{PoPy}, \cite{AdKeSm}). It would be of interest to have concrete examples of numbers $x \in [0,1)$ which are jointly $b$-normal (for some fixed $b$) and CF-normal.

Theorem \ref{thm1-intro} is a special case of a much more general result, which involves piecewise monotone maps. To give a flavor of this more general result we will formulate now its special case which involves not only Gauss transformation $T_G$ and times $b$ transformations $T_b$, but also $\beta$-transformations $T_{\beta}$. Recall that for $\beta >1$, $\beta \notin \mathbb{N}$,
$\beta$-transformation  $T_{\beta}$ is given by $T_{\beta} x = \beta x \bmod \, 1$. These transformations were introduced by R\'{e}nyi (\cite{Re}) who in particular showed that  there exists a unique $T_{\beta}$-ergodic probability measure $\mu_{\beta}$ which satisfies $1 - \frac{1}{\beta} \leq \frac{d \mu_{\beta}}{d \lambda} \leq \frac{1}{1 -1 / \beta}$.
 
\begin{Theorem}[Theorem \ref{Cor:pointwise}]
\label{Thm1.2:pje_TimesbGaussBeta}
Let $s, t \in \mathbb{N}$.
Let $b_1, \dots, b_s$ be distinct integers with $b_i \geq 2$ for all $i = 1, 2, \dots, s$ and let $\beta_1, \dots, \beta_t$ be distinct non-integers with $\beta_i > 1$ for all $i =1, 2, \dots, t$.
\begin{enumerate}
\item For any  $f_1, \dots, f_s, g_1, \dots, g_t \in L^{\infty} (\lambda)$,
\begin{equation}
\label{eqn1:Thm1.2}
\lim_{N \rightarrow \infty} \frac{1}{N} \sum_{n=0}^{N-1}  \prod_{i=1}^s f_i ({T_{b_i}^n x}) \prod_{j=1}^t g_j ({T_{\beta_j}^n x})  = \prod_{i=1}^s \int f_i \, d \lambda \, \prod_{j=1}^t \int g_j \, d \mu_{\beta_j}
\end{equation} 
for almost every $x \in [0,1]$.  
\item Under the assumption $\log \beta_j \ne \frac{\pi^2}{6 \log 2}$ for $1 \leq j \leq t$, 
for any  $f_0,  f_1, \dots, f_s, g_1, \dots, g_t \in L^{\infty} (\lambda)$,
\begin{equation} 
\label{eqn2:Thm1.2}
\lim_{N \rightarrow \infty} \frac{1}{N} \sum_{n=0}^{N-1} f_0 (T_G^n x) \prod_{i=1}^s f_i ({T_{b_i}^n x}) \prod_{j=1}^t g_j ({T_{\beta_j}^n x})  = \int f_0 \, d \mu_G \, \prod_{i=1}^s \int f_i \, d \lambda \, \prod_{j=1}^t \int g_j \, d \mu_{\beta_j}
\end{equation}
for almost every $x \in [0,1]$. 
\end{enumerate} 
\end{Theorem}

\begin{Remark}
The number $\frac{\pi^2}{6 \log 2}$ is the entropy of $T_G$. Also $\log b$ and $\log \beta$ are the entropies of $T_b$ and $T_{\beta}$ respectively. So, entropies of transformations in the above theorem are all distinct; this fact plays an important role in the proof of Theorem \ref{Thm1.2:pje_TimesbGaussBeta}. When $\log \beta = \frac{\pi^2}{6 \log 2}$, we do not know whether the limit of the averages
\[  \frac{1}{N} \sum_{n=0}^{N-1} f(T_G^n x) \, g(T_{\beta}^n x), \, N \in \mathbb{N} \]
exists almost everywhere or in $L^2$.
\end{Remark}

\begin{Remark}
Putting in formula \eqref{eqn1:Thm1.2} $g_j =1, (j = 1, 2, \dots, t)$ gives us the following statement    
\begin{equation}
\label{specialcase:eq}
\lim_{N \rightarrow \infty} \frac{1}{N} \sum_{n=0}^{N-1}  \prod_{i=1}^s f_i ({b_i^n x})  = \prod_{i=1}^s \int f_i \, d \lambda \quad \text{for a.e. } x.
\end{equation}
This result (which deals with commuting transformations having the same invariant measure) can be also derived from the work of Berend (see Theorem 5.1 in \cite{Be}).
A related result due to Maxfield \cite{Max} states, in its ergodic form, that for almost every $s$-tuple $(x_1, \dots, x_s) \in [0,1]^s$,
\begin{equation}
 \lim_{N \rightarrow \infty} \frac{1}{N} \sum_{n=0}^{N-1}  \prod_{i=1}^s f_i ({b_i^n x_i})  = \prod_{i=1}^s \int f_i \, d \lambda.
 \end{equation}
Maxfield's result follows immediately - via the pointwise ergodic theorem - from the ergodicity of the measure preserving transformation $T_{b_1} \times \cdots \times T_{b_s}$ acting on $[0,1]^s$. On the other hand, formula \eqref{eqn1:Thm1.2} (and even its special case \eqref{specialcase:eq}) is quite a bit less trivial and manifests a subtle ``diagonal phenomenon" which is also behind Furstenberg's approach to Szemer\'edi's theorem (see \cite{Fur2}). 
\end{Remark}

For a non-integer $\beta > 1$, any number $x \in [0,1)$ can be written as 
\begin{equation} 
\label{beta:representation}
x = \sum_{n=1}^{\infty} \frac{d_n}{\beta^n}, \, \, d_n \in \{ 0, 1, \dots, [\beta]\}.
\end{equation}
It is known that almost every $x$ in the interval $I_{\beta} = [0, \frac{1}{\beta - 1}]$ has a continuum of different representation of  the form \eqref{beta:representation}  (\cite{ErJK} and \cite{Sid}). 
We will be working with a special form of the representation \eqref{beta:representation} that was introduced in \cite{Re} under the name {\em $\beta$-expansion} of $x$ and is defined by the condition $d_n = [\beta T_{\beta}^{n-1} (x)]$ for $n \geq 1$.  

Following \cite{IS}, we say that $x$ is {\em $\beta$-normal} if for any finite string $S = (v_1, \dots, v_k) \in \{ 0, 1, \dots, [\beta]\}^k, k \in \mathbb{N}$,
\begin{equation}
\lim_{N \rightarrow \infty} \frac{ \# \{1 \leq n \leq N: d_n = v_1, \dots, d_{n+k-1} = v_k\}}{N} = \mu_\beta (C_{\beta}(S)),
\end{equation}
where $C_{\beta}(S)$ consists of numbers in $[0,1)$, whose first $k$ terms in the $\beta$-expansion are $(v_1, \dots, v_k)$.
As in the case of $b$-expansions, the pointwise ergodic theorem implies that almost every $x \in [0,1)$ is $\beta$-normal. 

\begin{Remark}
If $b$ is an integer with $b \geq 2$, for any finite string $S$ in $\{0, 1, \dots, b-1\}^k, k \in \mathbb{N}$, there is a number $x$ such that $S$ appears in the $b$-expansion of $x$, so $C_b(S) \ne \emptyset$. However, the situation for $\beta$-expansions is somewhat different. For instance, if $\beta = \frac{3}{2}$, then it is not hard to check that $C_{\beta} (S) = \emptyset$ for $S = (1, 1)$.  
\end{Remark}

We will introduce now some definitions which will allow us to formulate a corollary of Theorem \ref{Thm1.2:pje_TimesbGaussBeta} in terms of joint normality.

Let $s, t \in \mathbb{N}$. Let $b_1, b_2, \dots, b_s$ be distinct integers with $b_i \geq 2$ for all $i = 1, 2, \dots, s$ and let $\beta_1, \beta_2, \dots, \beta_t$ be distinct non-integers with $\beta_j > 1$ for all $j =1, 2, \dots, t$. 
Consider the following representations of $x \in [0,1)$: 
\begin{enumerate}
\item $b_i$-expansion ($1 \leq i \leq s$),
\[ x = \sum_{n=1}^{\infty} \frac{a_n^{(i)}}{b_i^n}, \quad a_n^{(i)} \in \{0, 1, \dots, b_i-1\},\]
\item $\beta_j$-expansion ($1 \leq j \leq t$),
\[x = \sum_{n=1}^{\infty} \frac{d_n^{(j)}}{\beta_j^n}, \quad d_n^{(j)} \in \{ 0, 1, \dots, [\beta_j]\},\]
\item continued fraction expansion
\[ x = \cfrac{1}{c_1+\cfrac{1}{c_2+\cfrac{1}{c_3 + \cdots}}} = [c_1, c_2, c_3, ...], \quad c_n \in \mathbb{N}.\]
\end{enumerate}

For $1 \leq i \leq s$, let $V_i = (v_1^{(i)}, \cdots, v_{k_i}^{(i)})$ be a finite string in  $\{0, 1, \dots, b_i -1\}^{k_i}$, for $1 \leq j \leq t$, let $W_j = (w_1^{(j)}, \cdots, w_{l_j}^{(j)})$ be a finite string in $\{0, 1, \dots, [\beta_j]\}^{l_j}$ and let $U = (u_1, \dots, u_m)$ be a finite string in $\mathbb{N}^m$.
For $x \in [0,1)$, let  
\begin{equation*}
    \begin{split}
        F_N^1 (x, V_1, \cdots, V_s, W_1, \cdots, W_t) =
 \# \{1 \leq n \leq N: \, &a_{n+p-1}^{(i)} = v_p^{(i)}, \, 1 \leq p \leq k_i, 1 \leq i \leq s,\\ &d_{n+q-1}^{(j)} = w_q^{(j)}, \, 1 \leq q \leq l_j,  1 \leq j \leq t \}.
    \end{split}
\end{equation*}
Similarly, for $x \in [0,1)$, let  
\begin{equation*}
    \begin{split}
      F_N^2(x, V_1, \cdots, V_s, W_1, \cdots, W_t, U)=   \# \{1 \leq n \leq N: \, &a_{n+p-1}^{(i)} = v_p^{(i)}, \, 1 \leq p \leq k_i, 1 \leq i \leq s, \\
        &d_{n+q-1}^{(j)} = w_q^{(j)}, \, 1 \leq q \leq l_j,  1 \leq j \leq t, c_{n+r-1}  = u_r, \, 1 \leq r \leq m \}.
    \end{split}
\end{equation*} 

We will say that $x \in [0,1)$ is jointly $(b_1, \dots, b_s, \beta_1, \cdots, \beta_t)$-normal if for any strings $V_1, \dots, V_s, W_1, \dots, W_t$,
 \[ \lim_{N \rightarrow \infty} \frac{F_N^1 (x, V_1, \cdots, V_s, W_1, \cdots, W_t)}{N} = 
 \prod_{i=1}^s \lambda (C_{b_i} (V_i)) \, \prod_{j=1}^t \mu_{\beta_j} (C_{\beta_j} (W_j)),\]
and is jointly $(b_1, \dots, b_s, \beta_1, \cdots, \beta_t, CF)$-normal if for any strings $V_1, \dots, V_s, W_1, \dots, W_t, U$, 
    \[ \lim_{N \rightarrow \infty} \frac{F_N^2(x, V_1, \cdots, V_s, W_1, \cdots, W_t, U)}{N}
    =  \prod_{i=1}^s \lambda (C_{b_i}(V_i)) \, \prod_{j=1}^t \mu_{\beta_j} (C_{\beta_j}(W_j)) \, \mu_G (C_G(U)),\]
where $C_{b_i} (V_i), C_{\beta_j} (W_j)$ and $C_G(U)$ are the so called {\em cylinder sets}, which were defined above and which correspond, respectively, to strings  $V_i$ in $b_i$-expansion, $W_j$ in $\beta_j$-expansion, and $U$ in continued fraction expansion.

We are now in a position to formulate a corollary of Theorem \ref{Thm1.2:pje_TimesbGaussBeta}, which pertains to joint normality with respect to $b$-expansions, $\beta$-expansions and continued fraction expansion.
\begin{Corollary}[Corollary \ref{Corollary4.4}]
\label{cor:normality:basic}
Let $s, t \in \mathbb{N}$. Let $b_1, \dots, b_s$ be distinct integers with $b_i \geq 2$ for all $i = 1, 2, \dots, s$ and let $\beta_1, \dots, \beta_t$ be distinct non-integers with $\beta_j > 1$ for all $j =1, 2, \dots, t$.
Then almost every $x \in [0,1)$ is jointly $(b_1, \dots, b_s, \beta_1, \dots, \beta_t)$-normal and also, if $\log \beta_j \ne \frac{\pi^2}{6 \log 2}$ for any $j= 1, 2 \dots, t$, almost every $x \in [0,1)$ is jointly $(b_1, \dots, b_s, \beta_1, \dots, \beta_t, CF)$-normal. 
\end{Corollary}

In this paper we will be mostly working with the class $\mathcal{T}$ which is comprised of ergodic measure preserving systems $(X, \mathcal{B}, \mu, T)$, where $\mu$ is a probability measure on Borel $\sigma$-algebra $\mathcal{B}$ of $X = [0, 1]$ such that for some $c \geq 1$, $\frac{1}{c} \mu (A) \leq \lambda (A) \leq c \mu (A)\,  \text{for all } A \in \mathcal{B}$  and $T: [0,1] \rightarrow [0,1]$ is a $\mu$-ergodic\footnote{Note that since distinct normalized ergodic $T$-invariant measures are mutually singular, in light of our assumptions such $T$-invariant measure $\mu$ is unique.}, piecewise monotone map, which is equipped with a (finite or countable) partition  $\mathcal{A}$ of $[0,1]$ into (open, closed, or half open on either end) intervals with nonempty and disjoint interiors such that for any interval $I \in \mathcal{A}$, the map $T|_{\text{int}(I)}$ is continuous and strictly monotone.
In addition, we will assume that the following conditions are satisfied: 
\begin{enumerate}
\item $\mathcal{B} = \bigvee_{j=0}^{\infty} T^{-j} \sigma( \mathcal{A}) \mod \lambda$, 
where $\sigma (\mathcal{A})$ denotes the sub-$\sigma$-algebra generated by $\mathcal{A}$. 
(i.e. partition $\mathcal{A}$ generates the $\sigma$-algebra $\mathcal{B}$.) 

\item The entropy of the partition $\mathcal{A}$ is finite: $H(\mathcal{A}) := - \sum_{I \in \mathcal{A}} \mu(I) \log \mu (I) < \infty$.  
\end{enumerate}

\begin{Remark}
Kolmogorov-Sinai entropy $h(T)$ is given by 
\[ h(T) = \sup \{h(T, \mathcal{C}) : \mathcal{C} \text{ is a measurable partition with } H(\mathcal{C}) < \infty \}, \]
where $h (T, \mathcal{C}) = \lim\limits_{n \rightarrow \infty} \frac{1}{n} H \left( \bigvee_{j=0}^{n-1} T^{-j} \mathcal{C} \right)$. 
It is well-known (cf. \cite{Pe} or \cite{Walt}) that if $H(\mathcal{C}) < \infty$, then $h(T, \mathcal{C}) \leq H(\mathcal{C})$ and  if  $\mathcal{A}$ generates $\sigma$-algebra $\mathcal{B}$, then $h(T) = h(T, \mathcal{A})$. So it follows that any system belonging to $\mathcal{T}$ has finite entropy. 
\end{Remark}
 
For convenience, we will be denoting members of $\mathcal{T}$ by $(T, \mu, \mathcal{A})$. 
Note that for any system which appears in Theorem \ref{Thm1.2:pje_TimesbGaussBeta}, there exists a natural partition $\mathcal{A}$ for which $(T, \mu, \mathcal{A}) \in \mathcal{T}$. Here is the list of these partitions.
\begin{enumerate}
\item $T_b x = bx \, \bmod 1$ ($b \in \mathbb{N}$, $b \geq 2$):
\begin{equation}
\label{paritition:b}
\mathcal{A}_b = \left\{ \left[0, \frac{1}{b} \right), \cdots, \left[ \frac{b-1}{b}, 1 \right)  \right\}. 
\end{equation}
\item $T_{\beta} x = \beta x \, \bmod  \, 1$ ($\beta >1$, $\beta \notin \mathbb{N}$):
\begin{equation} 
\label{paritition:beta}
\mathcal{A}_{\beta} = \left\{ \left[ 0, \frac{1}{\beta} \right), \cdots, \left[ \frac{[\beta]-1}{\beta}, \frac{[\beta]}{\beta} \right),  \left[ \frac{[\beta]}{\beta}, 1 \right)  \right\}.
\end{equation}
\item $T_G x = \frac{1}{x} \, \bmod \, 1$:
\begin{equation}
\label{paritition:G}
\mathcal{A}_G = \left\{  \left( \frac{1}{n+1}, \frac{1}{n} \right]: n \in \mathbb{N} \right\}.
\end{equation}
\end{enumerate}

We are going now to formulate a general result, which contains Theorem \ref{Thm1.2:pje_TimesbGaussBeta} as a special case.  Before doing so, we have to introduce  two technical conditions regarding the sequence of partitions $\mathcal{A}^n :=  \bigvee\limits_0^{n-1} T^{-j} \mathcal{A}$, $n \in \mathbb{N}$.

The first condition deals with entropy equipartition. The classical Shanon-McMillan-Breiman theorem (it will be discussed in detail in Subsection \ref{subsec3.1}) asserts that  for a ``typical" $A \in \mathcal{A}^n$, $\mu(A)$  is close to $e^{-n h(T)}$, where $h(T)$ is the entropy of $T$. More precisely, 
for any $\epsilon >0$, there is $N = N(\epsilon)$ such that if $n \geq N$, then 
\begin{equation} 
\label{SMB_eqn}
\mu \left( \{ x\in [0,1): x \in A \text{ for some } A \in \mathcal{A}^n \text{ with } e^{- n(h(T) + \epsilon)} \leq \mu(A) \leq e^{- n (h(T) - \epsilon )}    \}  \right) \geq  1 -  \epsilon.
\end{equation}

For our goals, we will need to postulate somewhat sharper condition: 
\begin{Definition}
$(T, \mu, \mathcal{A})$ has  {\em property E} if  for any $\epsilon >0$, there is a positive real number $c_0 = c_0 (\epsilon)$ such that, for any $n \in \mathbb{N}$,
\begin{equation}
\label{intro_eqn:propertyE}  
\mu \left( \{ x\in [0,1): x \in A \text{ for some } A \in \mathcal{A}^n \text{ with } e^{- n(h(T) + \epsilon)} \leq \mu(A) \leq e^{- n (h(T) - \epsilon )}    \}  \right) \geq  1 -  \frac{c_0}{n}.
\end{equation}

\end{Definition}
We observe that partitions $\mathcal{A}_b$, $\mathcal{A}_{\beta}$, $\mathcal{A}_G$ (see \eqref{paritition:b}, \eqref{paritition:beta}, and \eqref{paritition:G}) satisfy condition \eqref{intro_eqn:propertyE}.    
Indeed, since for any integer $b \geq 2$, $h(T_b) = \log b$, condition \eqref{intro_eqn:propertyE} is trivially satisfied for times $b$ maps $T_b$. As for the Gauss transformation $T_G$ and $\beta$-transformations $T_{\beta}$, 
the  details are provided in Subsection \ref{sec:appendix}.

Property M, which is  introduced in the following definition, can be interpreted as a variant of a strong mixing condition, which intertwines $T$-invariant measure $\mu$ with the iterations $\lambda \circ T^{-n}$ of Lebesgue measure $\lambda$ as $n \rightarrow \infty$.

\begin{Definition}
$(T, \mu, \mathcal{A})$ has  {\em property M} 
if there is a sequence $(b_n)_{n=1}^{\infty}$ of non-negative real numbers with $\sum\limits_{n=1}^{\infty} b_n < \infty$  such that for any natural number $l$, for any non-negative integer $n$, if $A \in \sigma (\mathcal{A}^l)$ and  $B \in \mathcal{B}$,
\begin{equation}
\label{propertyB:eqn} 
| \lambda (A \cap T^{-(n+l)} B)  - \lambda(A) \mu(B)| \leq b_n.\footnote{Conditions of this kind appear in the classical work of Khintchine (\cite{Kh}) and Gelfond (\cite{Ge}).}
\end{equation}
\end{Definition}

We will show in Subsection \ref{sec:appendix} below that the systems appearing in Theorem \ref{Thm1.2:pje_TimesbGaussBeta} have property M (in addition to having property E). 

We are now in a position to formulate  the main result of this paper. We will call the phenomenon encompassed in the following theorem {\em pointwise joint ergodicity}. It will be treated in detail in Section \ref{joint ergodicity}.
\begin{Theorem} [Theorem \ref{j.e.:pointwise}]
\label{intro-main-thm-pje}
Let $T_i: [0,1] \rightarrow [0,1], \, 1 \leq i \leq k$, be measurable transformations such that for each $i = 1, 2 \dots, k$, there is a partition $\mathcal{A}_i$ and a $T_i$-invariant ergodic measures $\mu_i$ with $(T_i, \mu_i, \mathcal{A}_i) \in \mathcal{T}$. 
Suppose that for each $i = 1, 2 \dots, k$, $(T_i, \mu_i, \mathcal{A}_i)$ satisfies properties E and M. If entropy values $h(T_1), \cdots, h(T_k)$ are distinct,
then for any $f_1, \dots, f_k \in L^{\infty}(\lambda)$,
\begin{equation}
\label{eq:Thm_Intro_pje} 
\lim_{N \rightarrow \infty} \frac{1}{N} \sum_{n=0}^{N-1}  f_1 ({T_1^n x})   f_2 ({T_2^n x}) \cdots  f_k ({T_k^n x}) = \prod_{i=1}^k \int f_i \, d \mu_i 
\end{equation}
for almost every $x \in [0,1]$.  
\end{Theorem}

\begin{Remark}
While the motivation for establishing Theorem \ref{intro-main-thm-pje} comes from the question related to joint normality, this theorem can also be viewed as a refinement of some results obtained in \cite{BeSo}, which pertain to $L^2$-joint ergodicity.
It is worth noting that the $L^2$-results obtained in \cite{BeSo} require less restrictive conditions on $T_i$ and, moreover, involve uniform Ces\`aro averages, i.e., have the form 
$$\lim\limits_{N-M \rightarrow \infty} \frac{1}{N-M} \sum_{n=M+1}^{N}   f_1 ({T_1^n x})   f_2 ({T_2^n x}) \cdots  f_k ({T_k^n x}) = \prod_{i=1}^k \int f_i \, d \mu_i.$$
\end{Remark}

 As it was already observed in the work of R\'{e}nyi (\cite{Re}) and Parry (\cite{Pa2}), ergodic theory provides a natural unifying framework for dealing with various representations of numbers. As we will see in Section \ref{Sec:FRep}, ergodic approach allows also to prove a rather general theorem which deals with various equivalent forms of normality and has as a quite special corollary the following theorem which summarizes various classical results pertaining to base $b$ normality.
 
\begin{Theorem} 
\label{intro:equiv:b-normal}
Let $b \geq 2$ be an integer. Let $x \in [0,1)$.
The following statements are equivalent:
\begin{enumerate}[(i)]
\item $x$ is normal in base $b$.
\item  $(b^n x)_{n \in \mathbb{N}}$ is uniformly distributed $\bmod \, 1$. 
\item For any fixed $m \geq 2$, $x$ is normal in base $b^m$. 
\item $b^k x$ is simply normal in base $b^m$ for all $k \geq 0$ and $m \geq 1$.
\item $x$ is simply normal in all of the bases $b, b^2, b^3, \cdots.$ 
\item There exists a sequence $(m_l)_{l \in \mathbb{N}}$ with $m_1 < m_2 < \cdots$ such that $x$ is simply normal in all of the bases $b^{m_l}$. 
\end{enumerate}
\end{Theorem}
\begin{Remark}
Borel's original characterization of base $b$ normality is (iv) (see \cite{Bor}). The equivalence of (i) and (iv) was shown by Niven and Zuckerman (\cite{NiZu}). 
The characterizations (v) and (vi) were established, respectively, by Pillai (\cite{Pil}) and Long (\cite{Lo}).  For more details, see Chapter 1, Section 8 in \cite{KN} and Appendix in \cite{BeDoMi}.
\end{Remark}

A general ergodic form of Theorem \ref{intro:equiv:b-normal} involves systems belonging to a rather large family of measure preserving systems which we will denote by $\tilde{\mathcal{T}}$ and which is defined by relaxing some of the conditions appearing in the definition of class $\mathcal{T}$. 
Class $\tilde{\mathcal{T}}$ is comprised of ergodic measure preserving systems $(X, \mathcal{B}, \mu, T)$, where $X = [0, 1]$, $\mathcal{B}$ is the Borel $\sigma$-algebra,
$\mu$ is a probability measure, which is equivalent to Lebesgue measure $\lambda$, and 
 $T: [0,1] \rightarrow [0,1]$ is a $\mu$-ergodic\footnote{Note that in light of our assumptions such $T$-invariant measure $\mu$ is unique.} map, which is equipped with a (finite or countable) partition  $\mathcal{A}$ of $[0,1]$ into (open, closed, or half open on either end) intervals with nonempty and disjoint interiors such that for any interval $I \in \mathcal{A}$, the map $T|_{\text{int}(I)}$ is continuous and strictly monotone.
In addition, we will assume that $\mathcal{B} = \bigvee_{j=0}^{\infty} T^{-j} \sigma( \mathcal{A}) \mod \lambda$, 
where $\sigma (\mathcal{A})$ denotes the sub-$\sigma$-algebra generated by $\mathcal{A}$ 
(i.e. partition $\mathcal{A}$ generates the $\sigma$-algebra $\mathcal{B}$.) 

As before, we will find it convenient to denote members of $\tilde{\mathcal{T}}$ by $(T, \mu, \mathcal{A})$. 

To formulate a general ergodic form of Theorem \ref{intro:equiv:b-normal} we utilize the notion of F-representation which was introduced by E. A. Robinson, Jr. in \cite{Rob}. 

Let  $(T, \mu, \mathcal{A}) \in \tilde{\mathcal{T}}$, where $\mathcal{A} = \{I_j: j \in D\}$, $D$ being a finite or countably infinite index set.
For $x \in X = [0,1)$, introduce a symbolic representation  $(r_n(x))_{n=1}^{\infty}$ by putting $r_n = r_n(x) = j$ if $T^{n-1} x \in I_j$. 
In \cite{Rob}, this representation is called F-representations and it is shown that if $T$ is topologically transitive (meaning that for  some $x \in X$ the forward orbit $\{T^n x: n \geq 0\}$ is dense in X) then the F-representation $(r_n(x))_{n=1}^{\infty}$ is unique for almost every $x$.
Note that if $(T, \mu, \mathcal{A}) \in \tilde{\mathcal{T}}$, then $T$ is transitive, since $T$ is ergodic with respect to $\mu$ (and $\mu$ is equivalent to $\lambda$). 

Write $\mathbf{T} = (T, \mu, \mathcal{A})$.
For a string $S = (v_1, \dots, v_k) \in D^k$, define the cylinder set $C_{\mathbf{T}}(S) (=  [v_1, \dots, v_k]_{\mathbf{T}}) := \{y \in [0,1): r_1(y) = v_1, \dots, r_n(y) =v_n\}$. 
Note that $\{  C_{\mathbf{T}}(S) : S \in D^k \} = \mathcal{A}^k  \, (\bmod \, \lambda)$ for any $k \in \mathbb{N}$.
We say that $x$ is 
\begin{itemize}
\item  {\em simply ${\mathbf{T}}$-normal} if for any $v$ in $D$, 
\[ \lim_{N \rightarrow \infty} \frac{\# \{ 1 \leq n \leq N : r_n = v \}}{N} =\mu([v]_{\mathbf{T}}).\]
\item {\em ${\mathbf{T}} $-normal} if for any $ k \in \mathbb{N}$ and any finite string $S = (v_1, \dots, v_k)$ in $D^k$, 
\[\lim_{N \rightarrow \infty} \frac{\#\{1 \leq n \leq N: r_n=v_1, \dots, r_{n+k-1} = v_k\}}{N} = \mu(C_{\mathbf{T}}(S)).\]  
\end{itemize}
For $m \in \mathbb{N}$, let ${\mathbf{T}_m} = (T^m, \mu, \mathcal{A}^m) \in \mathcal{T}$. 
Note that $x$ is {\em simply ${\mathbf{T}_m}$-normal} if $S = (v_1, \dots, v_m)$ in $D^m$, 
\[ \lim_{N \rightarrow \infty} \frac{\# \{ 1 \leq n \leq N : r_{m(n-1) +1} = v_1, \dots, r_{mn} = v_m \}}{N} =\mu(C_{\mathbf{T}}(S)).\]

Here is now a general ergodic form of Theorem \ref{intro:equiv:b-normal}.
\begin{Theorem}[Theorem \ref{equiv:normal:F-representation}]
\label{Intro:equiv:normal:F-representation} 
Let ${\mathbf{T}} = (T, \mu, \mathcal{A}) \in \tilde{\mathcal{T}}$ and assume that $T$ is totally ergodic.\footnote{ A measure preserving transformation $T$ is called {\em totally ergodic} if $T^k$ is ergodic for all $k \in \mathbb{N}$.} Let $x \in [0,1]$.
The following statements are equivalent.
\begin{enumerate}[(i)]
\item $x$ is ${\mathbf{T}}$-normal. 
\item $(T^n x)$ is equidistributed in $[0,1]$ with respect to $\mu$, that is, for any continuous function $f:[0,1] \rightarrow \mathbb{R}$,
\[ \lim_{N \rightarrow \infty} \frac{1}{N} \sum_{n=0}^{N-1} f(T^n x) = \int f \, d \mu.\]
\item For any $m \geq 2$, $x$ is ${\mathbf{T}}_m$-normal. 
\item $T^k x$ is simply ${\mathbf{T}}_m$-normal for all $k \geq 0$ and $m \geq 1$.
\item $x$ is simply ${\mathbf{T}}_m$-normal for all $m \geq 1$. 
\item There exists a sequence $(m_l)_{l \in \mathbb{N}}$ with $m_1 < m_2 < \cdots$ such that $x$ is simply ${\mathbf{T}}_{m_l}$-normal for any $l \in \mathbb{N}$.
\end{enumerate}
\end{Theorem}

One also has an analogue of Theorem  \ref{Intro:equiv:normal:F-representation} for the notion of joint normality. Before formulating this result, we have to introduce some definitions.

Let ${\mathbf{T}}^{(i)} = (T_i, \mu_i, \mathcal{A}_i) \in \tilde{\mathcal{T}}, \, i = 1, 2 \dots, k$, where $\mathcal{A}_i = \{I_{j}^{(i)}: j \in D_i\}$ and  $D_i$ is an index set, which may be finite or countably infinite.   
Let $(r_n^{(i)}(x))_{n=1}^{\infty}$ be F-representation of $x$ corresponding to $\mathbf{T}^{(i)}$, that is, $r_n^{(i)} = r_n^{(i)}(x) = j$ if $T_i^{n-1} x \in I_{j}^{(i)}$. 
We say that $x$ is 
\begin{itemize}
\item {\em jointly simply $({\mathbf{T}}^{(1)}, \dots, {\mathbf{T}}^{(k)})$-normal} if  for any $v_1, \dots, v_{k}$, where $v_i \in D_i$ $(1 \leq i \leq k)$,
\[\lim_{N \rightarrow \infty} \frac{\#\{1 \leq n \leq N: r_{n}^{(1)}=v_1, \dots, r_n^{(k)} = v_k\}}{N} = \prod_{i=1}^k \mu_i ([v_i]_{{\mathbf{T}}_i}).\]
\item {\em jointly $({\mathbf{T}}^{(1)}, \dots, {\mathbf{T}}^{(k)})$-normal} if for any finite strings $S_i = (v_1^{(i)}, \dots, v_{m_i}^{(i)})$ $(1 \leq i \leq k)$,
\[\lim_{N \rightarrow \infty} \frac{\#\{1 \leq n \leq N: r_n^{(i)}=v_1^{(i)}, \dots, r_{n+m_i-1}^{(i)} = v_{m_i}^{(i)} \text{ for } 1 \leq i \leq k \}}{N} = \prod_{i=1}^k \mu_i(C_{{\mathbf{T}}_i}(S_i)).\]  
\end{itemize}
As before, we will use the notation ${\mathbf{T}}_m^{(i)} = (T_i^m, \mu_i, \mathcal{A}_i^m)$. 
Note that $x$ is {\em jointly simply $({\mathbf{T}}_m^{(1)}, \dots, {\mathbf{T}}_m^{(k)})$-normal} if  for any finite strings $S_i = (v_1^{(i)}, \dots, v_{m}^{(i)}) \in D_i^{m}$ $(1 \leq i \leq k)$,
\[\lim_{N \rightarrow \infty} \frac{\#\{1 \leq n \leq N: r_{m (n-1) +1}^{(i)}=v_1^{(i)}, \dots, r_{ mn }^{(i)} = v_{m}^{(i)} \text{ for } 1 \leq i \leq k \}}{N} = \prod_{i=1}^k \mu_i(C_{{\mathbf{T}}_i}(S_i)).\]

\begin{Theorem}[Theorem \ref{equiv:jointnormal:F-representation}]
\label{intro:equiv:jointnormal:F-representation} 
Let ${\mathbf{T}}^{(i)}= (T_i, \mu_i, \mathcal{A}_i) \in \tilde{\mathcal{T}}, \, i = 1, 2 \dots, k$ and assume that $T_1 \times \cdots \times T_k$ is totally ergodic with respect to $\mu_1 \times \cdots \times \mu_k$.  Let $x \in [0,1]$.
The following statements are equivalent.
\begin{enumerate}[(i)]
\item $x$ is jointly $({\mathbf{T}}^{(1)}, \dots, {\mathbf{T}}^{(k)})$-normal. 
\item$(T_1^n x, \dots, T_k^n x)$ is equidistributed in $[0,1]^k$ with respect to $\mu_1 \times \cdots \times \mu_k$: for any continuous function $f$ on $[0,1]^k \rightarrow \mathbb{R}$,
\[ \lim_{N \rightarrow \infty} \frac{1}{N} \sum_{n=0}^{N-1} f(T_1^n x, \dots, T_k^n x) = \int f \, d (\mu_1 \times \cdots \times  \mu_k).\]
\item For any $m \geq 2$, $x$ is jointly $({\mathbf{T}}_m^{(1)}, \dots, {\mathbf{T}}_m^{(k)})$-normal. 
\item$x$ is jointly simply $({\mathbf{T}}_m^{(1)}, \dots, {\mathbf{T}}_m^{(k)})$-normal for all $m \in \mathbb{N}$. 
\item There exists a sequence $(m_l)_{l \in \mathbb{N}}$ with $m_1 < m_2 < \cdots$ such that $x$ is jointly simply $({\mathbf{T}}_{m_l}^{(1)}, \dots, {\mathbf{T}}_{m_l}^{(k)})$-normal for any $l \in \mathbb{N}$.
\end{enumerate}
\end{Theorem}

\begin{Remark}
It is worth noting an important difference between Theorems \ref{Intro:equiv:normal:F-representation} and \ref{intro:equiv:jointnormal:F-representation}  when $k >1$.
Namely,  while in light of the pointwise ergodic theorem, for any $\mathbf{T} =(T, \mu, \mathcal{A}) \in \tilde{\mathcal{T}}$, almost every $x$ in $X$ is $\mathbf{T}$-normal, the phenomenon of joint $({\mathbf{T}}^{(1)}, \dots, {\mathbf{T}}^{(k)})$-normality is much more delicate.
For example, the transformation $T_1 \times \cdots \times T_k$ can be totally ergodic even if for some $i \ne j$, $T_i = T_j$. 
However, in this case there are no jointly $({\mathbf{T}}^{(1)}, \dots, {\mathbf{T}}^{(k)})$-normal points. 
On the other hand, Theorem \ref{intro-main-thm-pje}, which deals with the phenomenon of pointwise joint ergodicity, implies that if each $\mathbf{T}^{(i)} = (T_i, \mu_i, \mathcal{A}_i) \in \mathcal{T}, (1 \leq i \leq k),$ has properties E and M and $h(T_1), \dots, h(T_k)$ are distinct, then  almost every $x$ is jointly $({\mathbf{T}}^{(1)}, \dots, {\mathbf{T}}^{(k)})$-normal. 
Similar results hold for mutually disjoint systems (see Section \ref{sec:disjoint} for the details). 
\end{Remark}

The structure of this paper is as follows.
In Section \ref{joint ergodicity}, we discuss the notion of  pointwise joint ergodicity and establish a characterization of pointwise joint ergodicity in terms of equidistribution. The results obtained in  this section will be instrumental for the proof of the main theorem of this paper (Theorem \ref{intro-main-thm-pje} above).
In Section \ref{sec:3}, we discuss properties E and M and describe numerous examples of systems which possess properties E and M and belong to the family $\mathcal{T}$. These examples are used in the subsequent sections to illustrate our results pertaining to joint normality. 
Section \ref{sec:pointwisejointergodicity} is devoted to the proof of Theorem \ref{intro-main-thm-pje}. 
In Section \ref{sec:disjoint}, we establish additional results on pointwise joint ergodicity which utilize the notion of disjointness. In Section \ref{Sec:FRep}, we deal with various equivalent forms of normality and joint normality, which were briefly discussed in the Introduction. 
In particular, we provide the proof of Theorem \ref{Intro:equiv:normal:F-representation}. 
We also discuss a variant of Theorem \ref{Intro:equiv:normal:F-representation} for joint normality. 

\section*{Acknowledgement}
Younghwan Son is supported by the NRF of Korea (NRF-2020R1A2C1A01005446).

\section{pointwise joint ergodicity}
\label{joint ergodicity}

In this section, we introduce the notion of pointwise joint ergodicity and present a useful characterization of pointwise joint ergodicity of transformations $T_1, \dots, T_k$ acting on measurable space $(X, \mathcal{B})$ in terms of the equidistribution in $X^k$ of the orbits $(T_1^n x, \dots, T_k^n x)$, $n = 0, 1, , 2 \dots.$

\begin{Definition} 
\label{Def:PJE}
Let $T_1, \dots, T_k: X \rightarrow X$ be measurable transformations on a measurable space $(X, \mathcal{B})$. Assume that for each $i \in \{1, 2, \dots, k\}$ there is a $T_i$-invariant probability measure $\mu_i$ on $\mathcal{B}$ so that for any $i,j \in \{1, 2, \dots, k\}$, the measures $\mu_i$ and $\mu_j$ are equivalent meaning that they are mutually absolutely continuous. 
We say that $T_1, \dots, T_k$ are {\em pointwise jointly ergodic} with respect to $(\mu_1, \dots, \mu_k)$ if for any $f_1, \dots, f_k \in L^{\infty}$,
$$\lim\limits_{N \rightarrow \infty} \frac{1}{N } \sum\limits_{n=0}^{N-1}  f_1(T_1^{n}x) f_2( T_2^n x)  \cdots f_k(T_k^n x) = \prod_{i=1}^k \int f_i \, d \mu_i \quad \text{a.e. } x. \footnote{Since $\mu_1, \dots, \mu_k$ are pairwise equivalent, we simply write ``$\text{a.e. } x$" instead of ``$\mu_i\text{-a.e. } x$ for $i = 1, ,2 \dots, k$".}$$
\end{Definition}
If measures $\mu_1, \dots, \mu_k$ are understood, we just say $T_1, \dots, T_k$ are pointwise jointly ergodic.

The following useful characterization of pointwise joint ergodicity is motivated by Lemma 2.1 in \cite{Be}. 
It will be utilized in the proof of our main theorem (Theorem \ref{intro-main-thm-pje}) in Section \ref{sec:pointwisejointergodicity}.

\begin{Theorem}
\label{lem:pje_gp}
Let $X$ be a compact metric space and let $\mathcal{B}$ be the Borel $\sigma$-algebra of $X$.
Suppose that $T_1, \dots, T_k: X \rightarrow X$ are measurable transformations on $(X, \mathcal{B})$ such that for each $i \in \{ 1, 2, \dots, k\}$ there is a $T_i$-invariant probability measure $\mu_i$ and for any $i, j \in \{ 1, 2, \dots, k\}$, the measures $\mu_i$ and $\mu_j$ are mutually absolutely continuous. Then the following conditions are equivalent.
\begin{enumerate}[(i)]
\item $T_1, \dots, T_k$ are pointwise jointly ergodic with respect to $(\mu_1, \dots, \mu_k)$.
\item For a.e. $x \in X$, the sequence of $k$-tuples $(T_1^n x, \dots, T_k^n x), \, n = 0, 1, 2, \dots,$ is equidistributed in $X^k$ with respect to $\mu_1 \times \cdots \times \mu_k$:
\[ \frac{1}{N} \sum_{n=0}^{N-1} \delta_{T_1^n x} \times \cdots \times \delta_{T_k^n x} \rightarrow \mu_1 \times \cdots \times \mu_k \text{ in weak* topology}. \]
\end{enumerate}
\end{Theorem}

\begin{proof} 
$(i) \Rightarrow (ii)$: Let $\mu = \mu_1 \times \cdots \times \mu_k$.
Denote $C(Y)$ by the space of continuous functions on a compact space $Y$.
We will use the classical fact that the functions of the form $f_1 \otimes  \cdots \otimes f_k (x_1, \dots, x_k) = \prod_{i=1}^k f_i(x_i)$, where $f_i \in C(X)$, are linearly dense in the space $C(X^k)$. 
Moreover, there is a countable set of functions $g_i \in C(X)$ such that their tensor product is linearly dense in $C(X^k)$. 
Thus one can find a countable dense subset $\{F_m: m \in \mathbb{N}\}$ of $C(X^k)$, where each $F_m$ is a finite sum of functions of the form $f_1 \otimes  \cdots \otimes f_k$. 

Condition (i) implies that for each $m \in \mathbb{N}$, there exists a full measure set $A_m$  such that if $x \in A_m$, then 
\begin{equation}
\label{eqn:Thm2.3:2.1}
\lim_{N \rightarrow \infty} \frac{1}{N} \sum_{n=0}^{N-1} F_m (T_1^n x, \cdots, T_k^n x) =  \int F_m \, d \mu.
\end{equation}
Let $A = \bigcap\limits_{m=1}^{\infty} A_m$.
It suffices to show that if $F$ is  a continuous function on $X^k$, then for any $x \in A$,
\begin{equation*}
\label{eqn:Thm2.3:2.2}
 \lim_{N \rightarrow \infty}  \frac{1}{N} \sum_{n=0}^{N-1} F (T_1^n x, \cdots, T_k^n x) = \int F \, d \mu.
 \end{equation*}
Let $\epsilon >0$ and let $F_m$ satisfy $\| F_m - F \|_{\sup} < \epsilon$. Then, clearly,
$$ \int |F_m - F| \, d \mu < \epsilon$$ 
and 
$$\frac{1}{N} \sum_{n=0}^{N-1} \left| F (T_1^n x, \cdots, T_k^n x) - F_m (T_1^n x, \cdots, T_k^n x) \right| < \epsilon.$$
So we have 
\begin{align*} 
\left| \frac{1}{N} \sum_{n=0}^{N-1} F (T_1^n x, \cdots, T_k^n x) - \int F \, d \mu \right|   
&\leq   \frac{1}{N} \sum_{n=0}^{N-1} \left| F (T_1^n x, \cdots, T_k^n x) - F_m (T_1^n x, \cdots, T_k^n x) \right| \\
 &\quad + \left| \frac{1}{N} \sum_{n=0}^{N-1} F_m (T_1^n x, \cdots, T_k^n x) - \int F_m \, d \mu \right|  
+ \int |F_m - F| \, d \mu \\
&\leq \left| \frac{1}{N} \sum_{n=0}^{N-1} F_m (T_1^n x, \cdots, T_k^n x) - \int F_m \, d \mu \right|  + 2 \epsilon.
\end{align*}
Letting $N \rightarrow \infty$ (and using \eqref{eqn:Thm2.3:2.1}) we get
\[ \limsup_{N \rightarrow \infty} \left| \frac{1}{N} \sum_{n=0}^{N-1} F (T_1^n x, \cdots, T_k^n x) - \int F \, d \mu \right| \leq 2 \epsilon.\]
Since $\epsilon$ is arbitrary, the desired result follows.

$(ii) \Rightarrow (i)$: Note first that condition (ii) implies that each $T_i$ is $\mu_i$-ergodic. Indeed, it follows from condition (ii) that there exists a set $X_0$ of full measure such that for any $x \in X_0$, for any $f \in C(X)$, and for any $i = 1, 2 \dots, k$,
\begin{equation*}
\label{eqn:erg}
\frac{1}{N} \sum_{n=1}^N f(T_i^n x) \rightarrow \int f \, d \mu_i.
\end{equation*}
So, by dominated convergence theorem, 
\begin{equation}
\label{eqn:erg:mean}
\frac{1}{N} \sum_{n=1}^N f(T_i^n x) \rightarrow \int f \, d \mu_i \text{ in } L^2 (\mu_i).
\end{equation}
It follows by the standard approximation argument that \eqref{eqn:erg:mean} holds for any $f \in L^2(\mu_i)$, which implies that each $T_i$ is $\mu_i$-ergodic.

We need to prove that for any bounded measurable functions $f_1, \dots, f_k$,
\[ \lim_{N \rightarrow \infty} \frac{1}{N} \sum_{n=0}^{N-1} \prod_{i=1}^k f_i(T_i^n x) = \prod_{i=1}^k \int f_i \, d \mu_i \quad \text{for a.e. } x.\]
Without loss of generality, we assume that $|f_i(x)| \leq 1$ for all $x \in X$ and for $i = 1, 2, \dots, k$. For any given $m \in \mathbb{N}$, there are continuous functions $g_1^{(m)}, \dots, g_k^{(m)}$ such that $|g_i^{(m)} (x)| \leq 1$ for all $x \in X$ and $\int |f_i - g_i^{(m)} | d \mu_i < \frac{1}{m}$. 

By condition (ii), there is a set of full measure $A_m^1$ such that if $x \in A_m^1$, then 
\begin{equation}
\label{eq:con2}
 \lim_{N \rightarrow \infty} \frac{1}{N} \sum_{n=0}^{N-1} \prod_{i=1}^k g_i^{(m)}(T_i^n x) = \prod_{i=1}^k \int g_i^{(m)} \, d \mu_i, 
 \end{equation} 
and by pointwise ergodic theorem, there is a set of full measure $A_m^2$ such that if $x \in A_m^2$, then for any $i = 1, 2, \dots, k$,
\begin{equation}
\label{eq:pje:sec2} 
\lim_{N \rightarrow \infty} \frac{1}{N}\sum_{n=0}^{N-1} |f_i(T_i^n x) - g_i^{(m)}(T_i^n x)| = \int |f_i - g_i^{(m)}| \, d \mu_i.
\end{equation}

Let $A_m = A_m^1 \cap A_m^2$.
We will show that if $x \in A_m$, then 
\begin{equation} 
\label{eq:2.5} 
\limsup_{N \rightarrow \infty}  \left|\frac{1}{N} \sum_{n=0}^{N-1} \prod_{i=1}^k f_i(T_i^n x) - \prod_{i=1}^k \int f_i \, d \mu_i \right| \leq \frac{2k}{m}.
\end{equation}
By triangle inequality,
\begin{align}
\label{eq:2.6}
\left|\frac{1}{N} \sum_{n=0}^{N-1} \prod_{i=1}^k f_i(T_i^n x) - \prod_{i=1}^k \int f_i \, d \mu_i \right| &\leq  
\frac{1}{N} \sum_{n=0}^{N-1} \left| \prod_{i=1}^k f_i(T_i^n x) - \prod_{i=1}^k g_i^{(m)}(T_i^n x) \right|  \\
&+  \left| \frac{1}{N} \sum_{n=0}^{N-1} \prod_{i=1}^k g_i^{(m)}(T_i^n x) - \prod_{i=1}^k \int g_i^{(m)} \, d \mu_i \right| \notag \\
&+ \left|\prod_{i=1}^k  \int g_i^{(m)} \, d \mu_i - \prod_{i=1}^k \int f_i \, d \mu_i \right| \notag
\end{align}

Let us estimate each of the three terms on the right hand side of \eqref{eq:2.6}.
To estimate the first term on the right hand side of \eqref{eq:2.6}, we use the following identity 
\begin{equation}
\label{identity:2.4}
    \prod_{i=1}^k a_i - \prod_{i=1}^k b_i = 
   (a_1 - b_1) b_2 \cdots b_k + a_1 (a_2 - b_2) b_3 \cdots b_k + \cdots + a_1 \cdots a_{k-1} (a_k - b_k).
\end{equation}

We have
$$ \left| \prod_{i=1}^k f_i(T_i^n x) - \prod_{i=1}^k g_i^{(m)}(T_i^n x) \right| \leq \sum_{i=1}^k |f_i(T_i^n x) - g_i^{(m)}(T_i^n x)|,$$ 
so  
\[ \frac{1}{N} \sum_{n=0}^{N-1} \left| \prod_{i=1}^k f_i(T_i^n x) - \prod_{i=1}^k g_i^{(m)}(T_i^n x) \right| \leq \sum_{i=1}^k \frac{1}{N} \sum_{n=0}^{N-1} |f_i(T_i^n x) - g_i^{(m)}(T_i^n x)|. \]
Thus, if $x \in A_m$, by \eqref{eq:pje:sec2},
\[ \limsup_{N \rightarrow \infty}\frac{1}{N} \sum_{n=0}^{N-1} \left| \prod_{i=1}^k f_i(T_i^n x) - \prod_{i=1}^k g_i^{(m)}(T_i^n x) \right|  \leq \sum_{i=1}^k \int |f_i - g_i^{(m)}| \, d \mu_i \leq  \frac{k}{m}.  \]
As for the second term on the right hand side of \eqref{eq:2.6},  we have by \eqref{eq:con2} that if $x \in A_m$,
\[ \lim_{N \rightarrow \infty} \left| \frac{1}{N} \sum_{n=0}^{N-1} \prod_{i=1}^k g_i^{(m)}(T_i^n x) - \prod_{i=1}^k \int g_i^{(m)} \, d \mu_i \right| = 0.\]
Finally applying \eqref{identity:2.4} to the last term on the right hand side of \eqref{eq:2.6}, we have that 
$$\left|\prod_{i=1}^k  \int g_i^{(m)} \, d \mu_i - \prod_{i=1}^k \int f_i \, d \mu_i \right| \leq \sum_{i=1}^k \int |g_i^{(m)} - f_i| \, d \mu_i \leq \frac{k}{m}.$$  
Putting all the things together, we get \eqref{eq:2.5}. 

Let $B = \bigcap_{m=1}^{\infty} A_m$. 
Clearly $B$ is of full measure (with respect to each $\mu_i$). It follows from \eqref{eq:2.5} that for any $x \in B$, we have 
\[ \lim_{N \rightarrow \infty}  \frac{1}{N} \sum_{n=0}^{N-1} \prod_{i=1}^k f_i(T_i^n x) = \prod_{i=1}^k \int f_i \, d \mu_i, \] 
which concludes the proof.
\end{proof}


\section{Number-theoretic transformations possessing properties E and M}
\label{sec:3}
In this section we discuss in detail properties E and M, which appear in the formulation our main theorem (Theorem \ref{j.e.:pointwise} in Section \ref{sec:pointwisejointergodicity}).
We also describe numerous examples of systems which possess properties E and M and belong to the family $\mathcal{T}$, which was introduced in Section \ref{sec:Intro}. These examples will be used in the subsequent sections to illustrate our results pertaining to pointwise joint ergodicity and joint normality. 

\subsection{Entropy and property E}
\label{subsec3.1}
\mbox{}

In this short section we formulate, for reader's convenience, the classical Shannon-McMillan-Breiman theorem as well as a corollary of it which serve as a motivation for introducing property E.

\begin{Theorem}[Shannon-McMillan-Breiman, see for example Section 6.2 in \cite{Pe}]
\label{entropy:SMB}
Let $(X, \mathcal{B}, \mu, T)$ be a probability measure preserving system and let $h(T)$ denote the Kolmogorov-Sinai entropy of $T$. 
Let $\mathcal{A}$ be a  countable measurable partition of $X$ with $- \sum_{A \in \mathcal{A}} \mu(A) \log \mu (A) < \infty$ and $  \bigvee_{n=0}^{\infty} T^{-n} \sigma( \mathcal{A} ) = \mathcal{B}  \mod \mu$. 
For $x \in X$, let $\mathcal{A}(x)$ be the element of $\mathcal{A}$ which contains $x$. Let $\mathcal{A}^n :=  \bigvee\limits_0^{n-1} T^{-j} \mathcal{A}$. Then
\[ -  \frac{1}{n} \log \mu (\mathcal{A}^n (x) ) \rightarrow h(T) \,\, \text{a.e. and in } L^1. \]
\end{Theorem}

The following corollary of Shannon-McMillan-Breiman theorem, which is often referred to as the entropy equipartition property, states roughly that 
for most sets $A \in \mathcal{A}^n$, $\mu(A)$ is approximately $e^{- n h(T)}$.

\begin{Corollary}
\label{entropy:intervals}
Given $\epsilon > 0$, there is $n_0$ such that for $n \geq n_0$ the elements of $\mathcal{A}^n$ can be divided into two classes 
$\mathfrak{g}^n = \mathfrak{g}^{n} (\epsilon)$ and $\mathfrak{b}^n= \mathfrak{b}^n (\epsilon)$ (``good" and ``bad" atoms) such that 
\begin{enumerate}
\item $\mu \left( \bigcup_{B \in \mathfrak{b}^n} B \right) < \epsilon$
\item For $A \in \mathfrak{g}^n$,
\begin{equation}
\label{entropy:equipartition} 
e^{- n (h(T) + \epsilon)} < \mu(A) < e^{-n (h(T) - \epsilon)}.
\end{equation}
\end{enumerate}
Thus we have that for $n \geq n_0$
\begin{equation} 
 \mu \left( \{ x\in [0,1): x \in A \text{ for some } A \in \mathcal{A}^n \text{ with } e^{- n(h(T) + \epsilon)} \leq \mu(A) \leq e^{- n (h(T) - \epsilon )}    \}  \right) \geq  1 -  \epsilon.
\end{equation}
\end{Corollary}

Note that, given $(T, \mu, \mathcal{A}) \in \mathcal{T}$ and $n \in \mathbb{N}$, $\mathcal{A}^n =  \bigvee\limits_0^{n-1} T^{-j} \mathcal{A}$ is also a finite or countable partition of $[0,1]$ into intervals; we will call these intervals {\em cylinders of rank $n$}.
Since $\frac{1}{c} \mu \leq \lambda \leq  c \mu$ for some $c \geq 1$, Corollary \ref{entropy:intervals} says that if a cylinder $A$ of rank $n$ belongs to the class $\mathfrak{g}^n$ of good atoms, then
\[ \frac{1}{c} e^{- n (h(T) + \epsilon)} \leq \lambda(A) \leq c e^{-n (h(T) - \epsilon)}. \]

The property E which we will recall here for the convenience of the reader can be viewed as a formalization of the requirement that most cylinders of rank $n$ belong to the class $\mathfrak{g}^n$ of good atoms.
\begin{Definition}
$(T, \mu, \mathcal{A}) \in \mathcal{T}$ has  {\em property E} if  for any $\epsilon >0$, there is a positive real number $c_0 = c_0 (\epsilon)$ such that, for any $n \in \mathbb{N}$,
\begin{equation}
\label{intro_eqn:propertyA}  
 \mu \left( \{ x\in [0,1): x \in A \text{ for some } A \in \mathcal{A}^n \text{ with } e^{- n(h(T) + \epsilon)} \leq \mu(A) \leq e^{- n (h(T) - \epsilon )}    \}  \right) \geq  1 -  \frac{c_0}{n}.
\end{equation}
\end{Definition}

\subsection{Property M and RU maps}
\label{sec3.2}
\mbox{}

In this subsection, we introduce the notion of a RU map and show that any weakly mixing RU map has property M. We also show that if $(T, \mu, \mathcal{A}) \in \mathcal{T}$ has property M, then $T$ is exact. 

We start with recalling the definition of property M.
\begin{Definition}
$(T, \mu, \mathcal{A}) \in \mathcal{T}$ has {\em property M} if there is a sequence $(b_n)_{n=1}^{\infty}$ of non-negative real numbers with $\sum\limits_{n=1}^{\infty} b_n < \infty$  such that, 
for any natural number $l$ and for any non-negative integer $n$, if $A \in \sigma (\mathcal{A}^l)$ and  $B \in \mathcal{B}$, then
\begin{equation}
\label{eq:propertyM}
| \lambda (A \cap T^{-(n+l)} B)  - \lambda(A) \mu(B)| \leq b_n.
\end{equation}
\end{Definition}

Define variation of $f$ on $I= [0,1]$ by 
\[ \text{Var} (f) := \sup \sum_i |f(x_i) - f(x_{i-1})|,\]
where the supremum ranges over all $0 = x_0 < x_1 < \cdots < x_n =1$. 

For every $f \in L^1(\lambda)$, set
$$\| f \|_{BV} := \| f \|_{\infty} + V(f), \, \text{where} \, V(f) = \inf \{\text{Var} (f^*): f^* = f \, \,\, \text{a.e. with respect to } \lambda\}. $$

Let $BV = :\{f \in L^1(\lambda): \|f\|_{BV} < \infty  \}$.
The space $BV$ is a Banach space with the norm $\|  \cdot \|_{BV}$. 

\begin{Definition}[cf. \cite{ADSZ},  p.19 and \cite{AaNa}, p.3] 
\label{def:RU}
A map $T:[0,1] \rightarrow [0,1]$ is said to be a {\em RU map}
if there exists a partition $\mathcal{A}$ of $[0,1]$ into (open, closed, or half open on either end) intervals with nonempty and disjoint interiors such that
\begin{enumerate}
    \item  $T$ is uniformly expanding: 
        \[ \inf \{ |T'(x)| : x \in \text{int}(A), A \in \mathcal{A} \}  > 1\]
    \item $T$ satisfies {\em Rychlik's condition}:
\[ \sum_{A \in \mathcal{A}} \| 1_{TA} v_A' \|_{BV} < \infty,\] 
where $v_{A}$, $A \in \mathcal{A}$, denotes the inverse of $T_{|A}: A \rightarrow TA$. 
\end{enumerate}
\end{Definition}

\begin{Remark}
 Notice that RU maps are presumed to be equipped with a partition $\mathcal{A}$ of $[0,1]$ (so that the properties (1) and (2) in Definition \ref{def:RU} are satisfied). 
 We will find it convenient to denote RU maps belonging to the class $\mathcal{T}$ by triples $(T, \mu, \mathcal{A})$ with the understanding that the partition $\mathcal{A}$ in the triple  $(T, \mu, \mathcal{A})$ corresponds to the partition which is stipulated in the Definition \ref{def:RU}.
\end{Remark}

The following result is motivated by Theorem 1(a) in the paper by Aaronson and Nakada (\cite{AaNa}).  
\begin{Theorem}
\label{RUmap:propertyM}
If $(T, \mu, \mathcal{A}) \in \mathcal{T}$ is a weakly mixing RU map, then it has property M. More precisely, there is a real number $r$ with $0 < r <1$ such that for any $A \in \sigma (\mathcal{A}^l)$ and $B \in \mathcal{B}$, 
\[ |\lambda(A \cap T^{-n-l} B) - \lambda(A) \mu (B)| \leq  \mu(B) \, O(r^n). \]
\end{Theorem}
The proof of Theorem \ref{RUmap:propertyM} is the almost same as that of Theorem 1(a) in the cited above paper by Aaronson and Nakada.
We provide the proof for reader's convenience.
Before embarking on the proof, we introduce some auxiliary facts. 
For $(T, \mu, \mathcal{A}) \in \mathcal{T}$,  let $\hat{T}: L^1(\lambda) \rightarrow L^1(\lambda)$ be the corresponding Frobenius-Perron operator. 
Then 
\[ \hat{T}^n f = \sum_{A \in \mathcal{A}^n}  1_{T^n A} \, v_A' \, f \circ v_A,\]
where $v_A: T^n A \rightarrow A$ is the inverse of the restriction of $T^n$ to $A$. 

We will need the following two results, which were obtained by Rychlik (\cite{Ry}) (see also formulas (1) and (2) on p.4 in \cite{AaNa}):
\begin{enumerate}
\item If $(T, \mu, \mathcal{A}) \in \mathcal{T}$ is a RU map, then there is $c_0 >0$ such that for any $f \in BV$,
\begin{equation}
\label{Rychlik:Prop1}
\sum_{A \in \mathcal{A}^l} \| 1_{T^l A} \, v_A' \, f \circ v_A \|_{BV}  \leq  c_0 \|f\|_{BV}, \quad  l \geq 1.
\end{equation}
\item If $T$ is a weakly mixing RU map,  then 
there exist $C >0$ and $r \in (0,1)$ so that for any $f \in BV$
\begin{equation}
\label{Rychlik} 
\| \hat{T}^n f - h \int_X f \, d \lambda \|_{BV} \leq C r^n \|f \|_{BV}, \quad  n \geq 1,
\end{equation}
where $h = \frac{d \mu}{d \lambda}.$
\end{enumerate}

\begin{proof}[Proof of Theorem \ref{RUmap:propertyM}]
Let $A \in \sigma(\mathcal{A}^{l})$. Write $A = \bigcup_i A_i$, where $A_i \in \mathcal{A}^l$. 
Note that
\[ \hat{T}^{l} 1_A = \sum_{I \in \mathcal{A}^l} 1_{T^l I} \, v_{I}' \, 1_A \circ v_{I} = \sum_{A_i} 1_{T^l A_i} \, v_{A_i}' \, 1 \circ v_{A_i}, \] 
where $v_{A_i}: T^l A_i \rightarrow A_i$ is given by $v_{A_i} := ({T^l}_{|A_i})^{-1}$. 
Then, by \eqref{Rychlik:Prop1} we get 
\begin{align*}
\| \hat{T}^l 1_A \|_{BV} 
&\leq  \sum_{A_i} \| 1_{T^l A_i} \, v_{A_i}' \, 1 \circ v_{A_i} \|_{BV} \\
&\leq \sum_{I \in \mathcal{A}^l} \| 1_{T^l I} \, v_{I}' \, 1 \circ v_{I} \|_{BV} \leq  c_0 \|1 \|_{BV} =  c_0
\end{align*}
for some $c_ 0 >0$.
Notice that 
$$\lambda (A \cap T^{-(n+l)}B) = \int ( \hat{T}^n (\hat{T}^l 1_A) ) \, 1_B \, d \lambda \quad \text{and} \quad 
 \lambda (A) \mu(B) = \int \left( h \int \hat{T}^{l} 1_A \, d \lambda \right) 1_B \,d \lambda.$$
Using \eqref{Rychlik} with $f = \hat{T}^l 1_A$, we obtain 
\[  \| \hat{T}^n (\hat{T}^l 1_A) - h \int \hat{T}^{l} 1_A \, d \lambda\|_{\infty} \leq  \| \hat{T}^n (\hat{T}^l 1_A) - h \int \hat{T}^{l} 1_A \, d \lambda\|_{BV} \leq Cr^n \| \hat{T}^l 1_A \|_{BV} \leq  c_0 C r^n. \]
Then,
\begin{align*}
 &|\lambda (A \cap T^{-(n+l)}B) - \lambda(A) \mu(B)| 
 = \left| \int \left( \hat{T}^n (\hat{T}^l 1_A) - h \int \hat{T}^{l} 1_A \, d \lambda \right) 1_B \, d \lambda \right| \\  
 &\quad  \leq  \| \hat{T}^n (\hat{T}^l 1_A) - h \int \hat{T}^{l} 1_A \, d \lambda\|_{\infty} \lambda (B) \leq c_0 C r^n \lambda(B) = \mu(B) O(r^n).
\end{align*}
\end{proof}

We conclude this section by proving that if $(T, \mu, \mathcal{A}) \in \mathcal{T}$ has property M, then it is exact.

A (non-invertible) measure preserving transformation $T$ on a probability space $(X, \mathcal{B}, \mu)$ is called  {\em exact endomorphism} if 
\[ \bigcap_{n=0}^{\infty} T^{-n} \mathcal{B} = \{ \emptyset, X\} \,\, (\bmod \, \mu).\]

Rokhlin (\cite{Rok}) showed that if a measure preserving transformation $T$ on a probability space $(X, \mathcal{B}, \mu)$ is exact, then $T$ has completely positive entropy (meaning that each of its non-trivial factors has positive entropy). 

\begin{Theorem}
\label{exactness}
If $(T, \mu, \mathcal{A}) \in \mathcal{T}$ has property M, then $T$ is exact. 
\end{Theorem}
\begin{proof}
Assume that $T$ is not exact. 
Then there is a measurable set $B \in \bigcap_{n=0}^{\infty} T^{-n} \mathcal{B}$ with $0 < \mu(B) < 1$ (and so $0 < \lambda(B) < 1$).
One can find a sequence of measurable sets $B_n \in \bigvee_{i=0}^{n-1} \sigma (T^{-i} \mathcal{A})$ such that $\lambda(B_n \triangle B) \rightarrow 0$ and $\lambda (B_n) \geq \epsilon_0$ for some $\epsilon_0 > 0$ if $n$ is sufficiently large.  
Since $B \in \bigcap_{n=0}^{\infty} T^{-n} \mathcal{B}$, for each $n \in \mathbb{N}$ there exists a measurable set $C_n \in \mathcal{B}$ such that $T^{-2n} C_n = B^c$. 
Note that 
\[ \lambda (B_n \cap T^{-2n} C_n) = \lambda (B_n \cap B^c) \leq \lambda (B_n \triangle B) \rightarrow 0.\]
By \eqref{eq:propertyM}, we have that 
\[ -b_n + \lambda (B_n \cap T^{-2n} C_n) \leq \lambda(B_n) \mu(C_n) \leq b_n + \lambda(B_n \cap T^{-2n} C_n).\]
Since $\lambda(B_n) \geq \epsilon_0$ for sufficiently large $n$, $\mu(C_n) \rightarrow 0$.
Then we have 
\[ 0 = \lim_{N \rightarrow \infty} \frac{1}{N} \sum_{n=1}^n \mu (C_n) = \lim_{N \rightarrow \infty} \frac{1}{N} \sum_{n=1}^n \mu (T^{-2n}C_n) = \mu(B^c),\]
so $\mu(B) =1$, which contradicts the assumption that $\mu(B) < 1$ 
\end{proof}

\subsection{Examples}
\label{sec:appendix}
In this subsection, we discuss some examples of number-theoretic maps which, along with an appropriate partition, belong to the class $\mathcal{T}$ and possess properties E and M. 

\subsubsection{Times $b$ maps: $T_b x = bx  \mod 1$, $b \geq 2$ an integer} 
Note that the map $T_b x = b x \, \bmod \, 1$ is $\lambda$-preserving and ergodic. 
It is well-known that the entropy $h(T_b)$ is  $\log b$. \\
Let 
\[ \mathcal{A}_b = \left\{ \left[0, \frac{1}{b} \right), \cdots, \left[ \frac{b-1}{b}, 1 \right)  \right\}. \]
Then clearly $(T_b, \lambda, \mathcal{A}_b) \in \mathcal{T}$. \\
It is also not hard to check that 
\begin{itemize}
\item for any $n \in \mathbb{N}$, if $I$ is a cylinder of rank $n$, then $\lambda(I) = \frac{1}{b^n}$, so the set $\mathfrak{b}^n$ of bad atoms is empty, which clearly implies that  $(T_b, \lambda, \mathcal{A}_b)$ has property E.
\item  if $A \in \sigma (\mathcal{A}_b^l)$, $l \in \mathbb{N},$ and $B \in \mathcal{B}$, then 
\[ \lambda (A \cap T_b^{-(n+l)} B) = \lambda (A) \lambda (B) \, \, \text{for any } n \geq 0,\]
so  $(T_b, \lambda, \mathcal{A}_b)$ has property M.
\end{itemize}

\subsubsection{Gauss map: $T_G x = \frac{1}{x} \, \bmod 1$}
It is well-known that Gauss map $T_G$ is ergodic with respect to Gauss measure $d \mu_G = \frac{1}{\log 2} \frac{dx}{1+x}.$  
Let
\begin{equation}
\mathcal{A}_G = \left\{  \left( \frac{1}{n+1}, \frac{1}{n} \right]: n \in \mathbb{N} \right\}.
\end{equation}
It is not hard to see that $(T_G, \mu_G, \mathcal{A}_G) \in \mathcal{T}$. 

To show that $(T_G, \mu_G, \mathcal{A}_G)$ has property E, recall first some basic facts about continued fractions.
Let $x \in [0,1)$ be an irrational number. Let $x = [x_1, x_2, x_3, \cdots ]$ be the continued fraction expansion of $x$. Let $p_n =p_n(x), q_n = q_n(x)$ denote the numerator and denominator of the $n$th convergent:
$$\frac{p_n}{q_n} = [x_1, x_2, \cdots, x_n].$$
Let $I_n(x)$ denote the cylinder set of rank $n$ containing $x$. The length of $I_n(x)$ is given by
\[ \lambda(I_n(x)) = \left| \frac{p_n}{q_n} - \frac{p_n + p_{n+1}}{q_n + q_{n+1}} \right| = \frac{1}{q_n (q_n + q_{n+1})}.\]
It is known (see, for instance, Section 14 in \cite{Kh} or formula (4.18) of Chapter 1 in \cite{Bi}) that for a.e. $x$,
\begin{equation}
\label{KLconst}
 \lim_{n \rightarrow \infty} \frac{1}{n} \log q_n(x) = \frac{\pi^2}{ 12 \log 2}.
 \end{equation}
Then Shannon-McMillan-Breiman Theorem (Theorem \ref{entropy:SMB}) implies that for a.e. $x$,
\[ h(T_G) =  \lim_{n \rightarrow \infty} - \frac{\log \mu_G (I_n(x))}{n}  
=  \lim_{n \rightarrow \infty} - \frac{\log \lambda(I_n(x))}{n}  = \lim_{n \rightarrow \infty} \frac{1}{n} \log q_n(x) (q_n(x) + q_{n-1} (x)) 
= \frac{\pi^2}{ 6 \log 2}.\]
The fact that $(T_G, \mu_G, \mathcal{A}_G)$ has property E is a consequence of the following theorem.
\begin{Theorem}[cf. Theorem 2 in \cite{FaWuShLi} or Theorem B in \cite{Ta}]
For any $\delta >0$, there exist $N > 0, B>0, \alpha >0$ such that for all $n \geq N$,
\begin{equation}
\label{eq:10}
\lambda \left\{ x: \left| \frac{1}{n} \log q_n(x) - \frac{\pi^2}{12 \log 2} \right| \geq \delta \right\} \leq B \exp (- \alpha n).
\end{equation}
\end{Theorem}

The following result implies that $T_G$ has property M.
\begin{Theorem}[cf. \cite{Kh}, \cite{Ku}, \cite{Le}]
\label{Khinchine}
 There is a constant $\gamma$ such that for $A \in \sigma (\mathcal{A}_G^l)$, $B \in \mathcal{B}$ and  for $n \in \mathbb{N}$
\begin{equation}
\label{CF:propertyB}  
\left| \lambda(A \cap T_G^{-n-l} B) - \lambda(A) \mu_G(B)  \right| =   \lambda(A) \, \mu_G(B) \, O(e^{- \gamma n}).
\end{equation}
\end{Theorem}

\subsubsection{$\beta$-transformations: $T_{\beta} x = \beta x \,  \bmod 1$, $\beta >1$ a non-integer} 
\label{betamap}
R\'{e}nyi showed that there is  a $T_{\beta}$-preserving and ergodic measure $\mu_{\beta}$ such that $1 - \frac{1}{\beta} \leq \frac{d \mu_{\beta}}{d \lambda} \leq \frac{1}{1 -1 / \beta}$. (See \cite{Re} and also see \cite{Pa2} and \cite{Ge}.) \\
Let
\[ \mathcal{A}_{\beta} = \left\{ \left[ 0, \frac{1}{\beta} \right), \cdots, \left[ \frac{[\beta]-1}{\beta}, \frac{[\beta]}{\beta} \right),  \left[ \frac{[\beta]}{\beta}, 1 \right)  \right\}. \]
Clearly $(T_{\beta}, \mu_{\beta}, \mathcal{A}_{\beta}) \in \mathcal{T}$. 

It is known (\cite{Pa1}) that $ h(T_{\beta}) = \log \beta$.
For $x \in[0,1)$, let $I_n(x)$ denote the cylinder of rank $n$ containing $x$.  
Then Shanon-McMillan-Breiman theorem implies that for a.e. $x$,
\[ \lim_{n \rightarrow \infty} - \frac{\log \lambda(I_n(x))}{n} = \lim_{n \rightarrow \infty} - \frac{\log \mu_{\beta} (I_n(x))}{n} =h(T_{\beta}) = \log \beta.\]

The fact that $(T_{\beta}, \mu_{\beta}, \mathcal{A}_{\beta})$ has property E is the consequence of the following theorem.
\begin{Theorem}
\label{proof_propertyA_beta}
For $\epsilon >0$,
\[ \lambda \left\{x:  \left| - \frac{\log \lambda(I_n(x))}{n} - \log \beta \right| > \epsilon \right\} \leq \frac{\beta}{\beta -1 } e^{- \epsilon n}.\]
\end{Theorem}
\begin{proof}
It is obvious that if $A$ is a cylinder of rank $n$, then $\lambda (A) \leq \frac{1}{\beta^n}$. In \cite{Re}, R\'{e}nyi  showed that number of cylinders of rank $n$ is less than or equal to $\frac{\beta}{\beta -1} \beta^n$. Then,
\[ \lambda \{x: x \in A \text{ for some cylinder }  A \text{ of rank } n  \text{ such that } \lambda(A) < \beta^{-n} e^{- \epsilon n} \} \leq \frac{\beta}{\beta -1} e^{- \epsilon n}.\]
\end{proof}
It is not hard to check that $(T_{\beta}, \mu_{\beta}, \mathcal{A}_{\beta})$ is a RU map. Also, it is known that $T_{\beta}$ is strongly mixing (see Section 4.5 in \cite{Rok1} or Theorem 1C in \cite{Ph1}), hence weakly mixing, so $(T_{\beta}, \mu_{\beta}, \mathcal{A}_{\beta})$ has property M by Theorem \ref{RUmap:propertyM}.

\subsubsection{Linear $\bmod \, 1$ transformations.} 

Times $b$ maps $T_b$ and $\beta$-transformations $T_{\beta}$ belong to a general family of maps $T_{\beta,\gamma}$:
$$T_{\beta, \gamma} x = \beta x + \gamma \bmod 1, \,\, (\beta >1, 0 \leq \gamma<1).$$
It is known that if either $\beta \geq 2, 0 \leq \gamma <1$ or $1 < \beta < 2, \gamma = 0$, then there is a $T_{\beta, \gamma}$-invariant and ergodic measure  $\mu_{\beta, \gamma}$ such that $1/c \leq \frac{d \mu_{\beta, \gamma}}{ d \lambda} \leq c$ for some $c \geq 1$. (See \cite{Hal} and \cite{Wil}.) 
Also, it is known  (\cite{Pa1}) that $ h(T_{\beta, \gamma}) = \log \beta.$
On the other hand, Parry showed that if $1 < \beta <2$ and $\gamma \ne 0$, then  $T_{\beta, \gamma}$ may not have  an invariant measure, which is equivalent to the Lebesgue measure. (See \cite{Pa2}). In what follows we will assume that either $\beta \geq 2, 0 \leq \gamma <1$, or $1 < \beta < 2, \gamma = 0$. 

Let 
\[ \mathcal{A}_{\beta, \gamma} = \left\{ \left[ 0, \frac{1 - \gamma}{\beta} \right),  \left[ \frac{1 - \gamma}{\beta},  \frac{2 - \gamma}{\beta} \right)\cdots, \left[ \frac{[\beta + \gamma] -1 - \gamma}{\beta}, \frac{[\beta + \gamma] - \gamma}{\beta} \right),  \left[ \frac{[\beta + \gamma] - \gamma}{\beta}, 1 \right)  \right\}. \]
One readily checks that $(T_{\beta, \gamma}, \mu_{\beta, \gamma}, \mathcal{A}_{\beta, \gamma}) \in \mathcal{T}$. 

Now let us show that $(T_{\beta, \gamma}, \mu_{\beta, \gamma}, \mathcal{A}_{\beta, \gamma})$ has properties E and M. 
(Note that we only need to consider the case that $\beta \geq 2, 0 \leq \gamma <1$.)

It is not hard to check that $(T_{\beta, \gamma}, \mu_{\beta, \gamma}, \mathcal{A}_{\beta, \gamma})$ is a RU map. Also it is known that $T_{\beta, \gamma}$ is strong mixing (indeed, it is exact (\cite{Wil})), hence weakly mixing,  so $(T_{\beta, \gamma}, \mu_{\beta, \gamma}, \mathcal{A}_{\beta, \gamma})$ has property M by Theorem \ref{RUmap:propertyM}.

One also can check that if $\beta \geq 2$, the number of cylinders of rank $n$ for $(T_{\beta, \gamma}, \mu_{\beta, \gamma}, \mathcal{A}_{\beta, \gamma})$  is $O(\beta^n)$; indeed, for $\beta =2$, the number of cylinders of rank $n$ is less than or equal to $\sum_{i=1}^n (2^i +1) = 2^{n+1} -1$ and for other cases ($\beta >2$), this follows from  Lemma 2.1 in \cite{Wil}. Then by arguing as in the proof of Theorem \ref{proof_propertyA_beta} in Subsection \ref{betamap}, one concludes that 
 $(T_{\beta, \gamma}, \mu_{\beta, \gamma}, \mathcal{A}_{\beta, \gamma})$ has property E.
 

\section{The proof of main theorem}
\label{sec:pointwisejointergodicity}
In this section we prove the main theorem of this article, which we formulate again for reader's convenience. 
Note that the properties E and M, which appear in the formulation, were introduced in Section \ref{sec:Intro}. (See also Section \ref{sec:3}.)

\begin{Theorem}[Theorem \ref{intro-main-thm-pje}]
\label{j.e.:pointwise}
Let $T_i: [0,1] \rightarrow [0,1], 1 \leq i \leq k$, be measurable transformations such that for each $i = 1, 2 \dots, k$, there is a partition $\mathcal{A}_i$ and a $T_i$-invariant ergodic measures $\mu_i$ with $(T_i, \mu_i, \mathcal{A}_i) \in \mathcal{T}$. 
Suppose that for each $i = 1, 2 \dots, k$, $(T_i, \mu_i, \mathcal{A}_i)$ satisfies properties E and M. If entropy values $h(T_1), \cdots, h(T_k)$ are distinct,
then $T_1, \dots, T_k$ are pointwise jointly ergodic with respect to $(\mu_1, \dots, \mu_k)$.  
\end{Theorem}

The following classical lemma will play an instrumental role in the proof of Theorem \ref{j.e.:pointwise}. 

\begin{Lemma}[See \cite{Har}, Lemma 1.5 and a historical remark preceding its formulation.]
\label{Lem:a.e.conv}
Let $(X, \mathcal{B}, \mu)$ be a probability space. Let $f_n(x)$ be a sequence of non-negative measurable functions on $X$. Let $u_n, v_n$ be sequences of real numbers such that \[ 0 \leq u_n \leq v_n.\] 
Write $V(N) := \sum\limits_{n=1}^N v_n$ and assume that $V(N) \rightarrow \infty$ as $N \rightarrow \infty$.
Suppose that there exists a constant $K$ such that for arbitrary large integers $M_1,M_2 \, (1 \leq M_1 < M_2)$,
\begin{equation} 
\label{eq:6}
\int_X \left(  \sum_{M_1 \leq n < M_2} (f_n(x) - u_n) \right)^2 \, d \mu \leq K \sum_{M_1 \leq n < M_2} v_n. 
\end{equation}
Then for any $\epsilon >0$ and for a.e. $x \in X$, 
\begin{equation}
\label{eq:7}
 \sum_{1 \leq n \leq N} f_n(x) = \sum_{1 \leq n \leq N} u_n + O \left(V (N)^{1/2} (\log (V(N) + 2))^{2/3 + \epsilon} + \max_{1 \leq n \leq N} u_n \right).
 \end{equation}
\end{Lemma}

\begin{proof}[Proof of Theorem \ref{j.e.:pointwise}]
Recall that $\mathcal{A}_i^m$ is a finite or countable partition of the unit interval $[0,1]$ and the elements of $\mathcal{A}_i^m$ are intervals, which are called cylinders of rank $m$ (corresponding to $(T_i, \mu_i, \mathcal{A}_i)$).
It is enough to show that for any $l \in \mathbb{N}$ and for any $A_i \in \mathcal{A}_i^{l}$, $i = 1, 2, \dots, k$, 
\begin{equation}
\label{eq:a.e.conv} 
\lim_{N \rightarrow \infty} \frac{1}{N} \sum_{n=0}^{N-1} \prod_{i=1}^k 1_{A_i} (T_i^n x) = \prod_{i=1}^k \mu_i(A_i) \quad \text{ for almost every } x.
\end{equation}
To see that the validity for \eqref{eq:a.e.conv} implies the result in question observe first that since the family of cylinders is countably infinite, \eqref{eq:a.e.conv} implies that for almost every $x$, for any $A_i \in \mathcal{A}_i^{l}$, $i = 1, 2, \dots, k$, 
\begin{equation}
\label{eq:a.e.conv2} 
\lim_{N \rightarrow \infty} \frac{1}{N} \sum_{n=0}^{N-1} \prod_{i=1}^k 1_{A_i} (T_i^n x) = \prod_{i=1}^k \mu_i(A_i).
\end{equation}
Now, since $\mathcal{B} = \bigvee_{m=1}^{\infty} \sigma (\mathcal{A}_i^m)$ for all $i = 1, 2, \dots, k$, 
any interval $J \subset [0,1]$ can be approximated by cylinders. 
Hence, by the standard approximation argument,
 the fact that \eqref{eq:a.e.conv2} holds for any $A_i \in \mathcal{A}_i^l, i = 1, 2, \dots, k, l \in \mathbb{N}$ implies that it holds for any intervals $J_1, \dots, J_k$:
\begin{equation*}
\lim_{N \rightarrow \infty} \frac{1}{N} \sum_{n=0}^{N-1} \prod_{i=1}^k 1_{J_i} (T_i^n x) = \prod_{i=1}^k \mu_i(J_i). 
\end{equation*}
This, in turn, implies that for almost every $x$, $(T_1^n x, \dots, T_k^n x) $ is equidistributed with respect to $\mu_1 \times \cdots \times \mu_k$. The pointwise joint ergodicity of $T_1, \dots, T_k$ follows then from Theorem \ref{lem:pje_gp}.

We will derive \eqref{eq:a.e.conv} by applying Lemma \ref{Lem:a.e.conv} to 
\[ f_n(x) = \prod_{i=1}^k 1_{A_i} (T_i^n x), \quad u_n = \prod_{i=1}^k \mu_i(A_i), \quad v_n =  \max (1,  \log n). \]
(Note that $f_n(x)$ and $v_n$ depend on $n$, but $u_n$ is a constant sequence, so we can ignore the last term $\max\limits_{1 \leq n \leq N} u_n$ in \eqref{eq:7}.)

Write $h_i = h(T_i)$ for simplicity. Without loss of generality we assume that $h_1 < h_2 < \cdots < h_k$. Choose  $\epsilon, \delta, \gamma > 0$ small enough so that for $i = 2, \dots, k$, 
\begin{equation}
\label{gamma}
h_i - (1+\delta)h_{i-1}  - (2+ \delta) \epsilon \geq \gamma.
\end{equation}

Since $(T_i, \mu_i, \mathcal{A}_i) \in \mathcal{T}$ for $i=1, 2, \dots, k$, there exists $c \geq 1$ such that $\frac{1}{c} \mu_i \leq \lambda \leq c \mu_i$ for all $i = 1, 2, \dots, k$.
For $M_1, M_2 \in \mathbb{N}$ with $M_1 < M_2 \leq (1+\delta)M_1$, let 
$$I_{M_1, M_2} = \{(m,n): M_1+l \leq m < n \leq M_2 -l, n- m \geq l \} \subset \mathbb{N}^2.$$

We will show that  \eqref{eq:a.e.conv} is the consequence of the following two claims:
 
Claim 1. For $(m,n) \in I_{M_1, M_2}$,
\begin{equation}
\label{estimate:pje1}
\lambda \left( \bigcap_{i=1}^k \left(T_i^{-m} A_i \bigcap T_i^{-n} A_i \right) \right) = \prod_{i=1}^k \mu_i(A_i)^2 + c_{m,n},
\end{equation}
where $\sum\limits_{(m,n) \in I_{M_1, M_2}} |c_{m,n}| = O(M_2 - M_1)$

Claim 2. For $M_1+l \leq m \leq M_2 -l$,
\begin{equation}
\label{estimate:pje2}
\lambda \left(\bigcap_{i=1}^k T_i^{-m} A_i \right) = \prod_{i=1}^k \mu_i(A_i) + d_{m},
\end{equation}
where $\sum\limits_{m = M_1 +l}^{M_2 - l} |d_{m}| = O(1)$. 

Before proving the above claims, let us show that \eqref{eq:a.e.conv} follows from \eqref{estimate:pje1} and \eqref{estimate:pje2}.
First, notice that \eqref{estimate:pje1} and \eqref{estimate:pje2} imply that
\begin{align*} 
&\int_{0}^1 \left( \sum_{n=M_1}^{M_2} \left( \prod_{i=1}^k 1_{A_i} (T_i^{n} x) - \prod_{i=1}^k \mu_i(A_i)\right)  \right)^2 \, d \lambda \\
&= 2 \int_0^1 \sum_{(m,n) \in I_{M_1, M_2}} \left( \prod_{i=1}^k 1_{A_i} (T_i^{m} x) - \prod_{i=1}^k \mu_i(A_i)\right) \left( \prod_{i=1}^k 1_{A_i} (T_i^{n} x) - \prod_{i=1}^k \mu_i(A_i)\right) \, d \lambda   + O (M_2 -M_1) \\
&= O \left( \sum_{(m,n) \in I_{M_1, M_2}} |c_{m,n}| \right)  + O \left( \sum_{m = M_1+l}^{M_2-l} |d_m| \right) (M_2 - M_1) + O(M_2 - M_1) \\
& = O (M_2 -M_1).
\end{align*}

Given any $M < N$, partition the interval $[M, N]$ into $O(\log (N-M))$ intervals of the form $[M_1, M_2]$ with $M_2 \leq (1+\delta)M_1$. Then we have 
\begin{equation}
\label{goal:thm-pointwise}
 \int_{0}^1 \left( \sum_{n=M+1}^N  \left(  \prod_{i=1}^k 1_{A_i} (T_i^{n} x) - \prod_{i=1}^k \mu_i(A_i) \right) \right)^2 \, d \lambda = O ((N -M) \log (N-M)). 
 \end{equation}
Note also that 
\begin{align*}
    (N-M) \log (N-M) &\leq (N-M)\log N = N \log N - M \log N \\ 
    &\leq N \log N - M \log M = O \left(\sum\limits_{n=M}^N \log n \right) 
    = O \left(\sum\limits_{n=M}^N v_n \right).
\end{align*}
 So Lemma \ref{Lem:a.e.conv} implies
 \begin{equation*}
 \label{eq4.9}
 \sum_{1 \leq n \leq N} \prod_{i=1}^k 1_{A_i} (T_i^n x) = \sum_{1 \leq n \leq N} \prod_{i=1}^k \mu_i(A_i) + O \left( (N \log N)^{1/2} (\log (N \log N + 2))^{2/3 + \epsilon}  \right),
 \end{equation*}
 which, in turn, gives \eqref{eq:a.e.conv}.  

Now let us prove Claim 1 (the proof of Claim 2 is analogous and is omitted). 
We start with two preliminary observations. In what follows, we tacitly assume that $\epsilon, \delta, \gamma$ satisfy \eqref{gamma}.
\begin{enumerate}
\item By Corollary \ref{entropy:intervals}, for any $i=1, 2, \dots, k$ and for all $n \in \mathbb{N}$, we have a partition
\[ \mathcal{A}_i^n   = \mathfrak{g}_i^n  \cup \mathfrak{b}_i^{n} \]
such that 
\begin{enumerate}
\item $\mathfrak{g}_i^n = \{ I \in \mathcal{A}_i^n : \frac{1}{c} e^{- n (h_i + \epsilon)} \leq \lambda (I) \leq c e^{-n (h_i - \epsilon)} \}$
\item $\mathfrak{b}_i^n = \mathcal{A}_i^n  \setminus \mathfrak{g}_i^n$.
\end{enumerate}
Since $(T_i, \mu_i, \mathcal{A}_i)$ has property E, there is a real number $a = a(\epsilon) > 0$ so that for $i = 1, 2, \dots, k$,
\begin{equation}
\label{constant-a}
\lambda \left( \bigcup_{B \in \mathfrak{b}_i^n} B \right) \leq \frac{a}{n}.
\end{equation} 
\item Since $(T_i, \mu_i, \mathcal{A}_i)$ has property M, there exists a sequence $(b_n)_{n=1}^{\infty}$ of non-negative real numbers  with $\sum\limits_{n=1}^{\infty} b_n < \infty$, such that for $i = 1, 2, \dots, k$,
 if $A \in \sigma (\mathcal{A}_i^l)$ and $B \in \mathcal{B}$,
\[ | \lambda (A \cap T_i^{-(n+l)} B)  - \lambda(A) \mu_i(B)| \leq b_n \text{ for any } n \in \mathbb{N}.\]
\end{enumerate}

To show \eqref{estimate:pje1}, we will need the following result: 

Claim 3: For $(m,n) \in I_{M_1, M_2}$, there exist sets $I_1 (k)$ and $I_2(k)$ such that 
\[  \bigcap_{i=1}^k (T_i^{-m} A_i \cap T_i^{-n} A_i) =  I_1(k) \cup I_2(k) \]
where
\begin{enumerate}
\item $I_1(k)$ is union of some cylinders in $\mathfrak{g}_k^{n+l}$ and satisfies 
$$ \lambda (I_1(k))  = \prod\limits_{i=1}^k \mu_i (A_i)^2 + O \left( b_{m-M_1-l} + b_{n-m-l}  + \frac{1}{M_1} + e^{- \gamma M_1} \right),$$
\item $I_2(k)$ is a set, which is disjoint from $I_1(k)$, and satisfies $\lambda (I_2(k)) = O \left( \frac{1}{M_1} +  e^{- \gamma M_1} \right)$.
\end{enumerate}

Note that if Claim 3 holds, then we have  
\[|c_{m,n}| = O \left(   b_{m-M_1-l} + b_{n-m-l}  + \frac{1}{M_1} +  e^{- \gamma M_1}  \right)\] 
and so \[ \sum_{(m,n) \in I_{M_1, M_2}} |c_{m,n}|= O(M_2 - M_1),\]
which implies Claim 1.
We prove Claim 3 by induction. 
Let us first show that there exist sets $I_1 (1)$ and $I_2(1)$ such that 
\[ T_1^{-m} A_1 \cap T_1^{-n} A_1 =  I_1(1) \cup I_2 (1),\]
where 
\begin{itemize}
\item $I_1(1)$ is union of some intervals in $ \mathfrak{g}_1^{n+l}$ and satisfies
$$\lambda(I_1(1)) = \mu_1(A_1)^2 + O \left(b_{m-M_1-l} + b_{n-m-l}  + \frac{1}{M_1} \right)$$
\item $I_2(1)$ is  a set, which is disjoint from $I_1(1)$, and satisfies $\lambda(I_2) =  O \left(\frac{1}{M_1} \right)$.
\end{itemize}

Indeed, note that $T_1^{-m} A_1 \cap T_1^{-n} A_1 \in \sigma(\mathcal{A}_1^{n+l})$, so
\begin{itemize}
\item $I_1(1)$ is union of some intervals in $\mathfrak{g}_1^{n+l}$, more precisely, $I_1(1) = \bigcup_{i=1}^{s(1)} J_i(1)$ such that $J_i(1) \in  \mathfrak{g}_1^{n+l}$ for $1 \leq i \leq s(1)$, 
\item $I_2(1)$ is union of some intervals in $\mathfrak{b}_1^{n+l}$, so $\lambda (I_2 (1)) \leq \frac{a}{n+l} = O \left(\frac{1}{M_1} \right)$. 
\end{itemize}
Moreover, by property M of $(T_1, \mu_1, \mathcal{A}_1)$, we have that
\begin{align*}
\lambda (I_1(1)) &= \lambda (T_1^{-m} A_1 \cap T_1^{-n} A_1) + O \left( \lambda (I_2(1)) \right) \\
&= \lambda (T_1^{-m}A_1 \cap T_1^{- ((n-m-l) +(m+l))} A_1) + O \left( \frac{1}{M_1} \right)\\
&= \mu_1(T_1^{-m}A_1) \cdot \mu_1(A_1) + O \left( b_{n-m-l} + \frac{1}{M_1} \right) \quad \quad (\text{by property M}) \\
&= \mu_1 (X \cap T_1^{-(m -M_1 -l) + (M_1 +l) } A_1) \cdot \mu_1(A_1) + O \left( b_{n-m-l} + \frac{1}{M_1} \right) \\
&= \mu_1(A_1)^2 + O  \left( b_{m-M_1-l} + b_{n-m-l} + \frac{1}{M_1} \right) \quad \quad  (\text{by property M})
\end{align*}

To continue the induction, suppose that for $1 \leq q < k$, there are sets $I_1(q)$ and $I_2(q)$ such that
\[  \bigcap_{i=1}^q (T_i^{-m} A_i \cap T_i^{-n} A_i) =  I_1(q) \cup I_2(q) \]
where for some $s(q)$
\begin{itemize}
\item $I_1(q) = \bigcup\limits_{i=1}^{s(q)} J_i(q)$, where $J_i(q) \in \mathfrak{g}_q^{n+l}$ for $1 \leq i \leq s(q)$, and 
$$ \lambda (I_1(q))  = \prod\limits_{i=1}^q \mu_i (A_i)^2 + O \left( b_{m-M_1-l} + b_{n-m-l}  + \frac{1}{M_1} + e^{- \gamma M_1} \right),$$
\item $I_2(q)$ is a set, which is disjoint from $I_1(q)$, and satisfies $\lambda (I_2(q)) = O \left( \frac{1}{M_1} +  e^{- \gamma M_1} \right)$.
\end{itemize}
We have to show that there are sets $I_1(q+1)$ and $I_2(q+1)$ such that 
\[  \bigcap_{i=1}^{q+1} (T_i^{-m} A_i \cap T_i^{-n} A_i) =  I_1(q+1) \cup I_2(q+1) \]
where
\begin{itemize}
\item $I_1(q+1) = \bigcup\limits_{i=1}^{s(q+1)} J_i(q+1)$, where $J_i(q+1) \in \mathfrak{g}_{q+1}^{n+l}$ for $1 \leq i \leq s(q+1)$, and
$$\lambda (I_1(q+1)) = \prod\limits_{i=1}^{q+1} \mu_i (A_i)^2 + O \left( b_{m-M_1-l} + b_{n-m-l}  + \frac{1}{M_1} +  e^{- \gamma M_1} \right),$$
\item $I_2(q+1)$ is a set, which is disjoint from $I_1(q+1)$, and satisfies $\lambda (I_2(q+1)) = O \left( \frac{1}{M_1} +  e^{- \gamma M_1} \right)$,
\end{itemize}

To conclude the proof, we need first to establish the following statement. 

Claim 4.
Given $I_1 (q) = \bigcup\limits_{i=1}^{s(q)} J_i(q)$ as above, one can write
\[ I_1(q) = H_1(q) \cup H_2(q),\]
where  
\begin{itemize}
\item $H_1(q) $ is union of some cylinders in $ \mathfrak{g}_{q+1}^{M_1}$,
\item $\lambda (H_2(q)) = O (\frac{1}{M_1} + e^{-\gamma M_1})$.
\end{itemize}

{\em Proof of  Claim 4.} 
Notice that  $J_i(q) \in \mathfrak{g}_q^{n+l}$, so  
\begin{equation*}
\lambda (J_i(q)) \geq \frac{1}{c} \mu(J_i(q)) \geq \frac{1}{c} e^{-(n+l) (h_q + \epsilon)} \geq \frac{1}{c} e^{- M_2 (h_q + \epsilon)}.
\end{equation*}
Since $\sum\limits_{i=1}^{s(q)} \lambda (J_i(q))  = \lambda \left( \bigcup\limits_{i=1}^{s(q)} (J_i(q)) \right) \leq 1$, we have $s(q) \frac{1}{c} e^{- M_2 (h_q + \epsilon)} \leq 1$, which implies
\begin{equation}
\label{estimate:$s(q)$:sec4} 
s(q) \leq c e^{M_2 (h_q + \epsilon)}.
\end{equation}

Assume that $\mathfrak{g}_{q+1}^{M_1}$ is comprised of intervals $U_1, \dots, U_s$, where  $U_i$ are subintervals of the interval $[0,1]$ having disjoint interiors. If $\mathfrak{b}_{q+1}^{M_1} = \emptyset$, then $U_1, \dots, U_s$ form a partition of the unit interval $[0,1]$.
Otherwise, there exist intervals $V_1, \dots, V_t$ such that $U_1, \dots, U_s, V_1, \dots V_t$ form a partition $\mathcal{P}$ of $[0,1]$. 
For each interval $J_i(q)$,  there exist a finite family of intervals $I_{i,j}$ $(1 \leq j \leq t_i)$ and two intervals $K_{i,1}$ and $K_{i,2}$ in $\mathcal{P}$ such that 
\[ \bigcup_{j=1}^{t_i} I_{i,j} \subset \overline{J_i(q)} \subset \left( \bigcup_{j=1}^{t_i} \overline{I_{i,j}} \right) \bigcup \left( \overline{ K_{i,1}} \cup \overline{K_{i,2}} \right). \]

Let $H_1(q)$ be the union of intervals $I_{i,j}, 1 \leq i \leq s(q), 1 \leq j \leq t_i$ which belong to $\mathfrak{g}_{q+1}^{M_1}$. Let $H_2(q) = I_1(q) \setminus H_1(q)$. 
Let us show that 
$$\lambda (H_2(q)) = O \left( \frac{1}{M_1} + e^{-\gamma M_1} \right).$$
Note that $H_2(q)$ is covered by at most $2s(q)$  intervals from $\mathfrak{g}_{q+1}^{M_1}$ and - possibly - by  some intervals from $ \mathfrak{b}_{q+1}^{M_1}$. 
Since $\lambda (J) \leq c e^{-M_1 (h_{q+1} - \epsilon)}$ for any cylinder $J \in \mathfrak{g}_{q+1}^{M_1}$,
\begin{equation}
\label{eq:estH}
\lambda(H_2(q)) \leq \lambda \left(\bigcup_{B \in \mathfrak{b}_{q+1}^{M_1}} B \right) + 2 s(q) c e^{-M_1 (h_{q+1} - \epsilon)}.
\end{equation}
By \eqref{gamma},
\begin{equation*}
-M_1 (h_{q+1} - \epsilon) +M_2 (h_q + \epsilon) \leq -M_1 (h_{q+1} - (1+ \delta)h_q - (2+ \delta) \epsilon) \leq - \gamma M_1,
\end{equation*}
and so by \eqref{estimate:$s(q)$:sec4}
\begin{equation}
\label{eq4.13}
2 s(q) c e^{-M_1 (h_{q+1} - \epsilon)} \leq 2 c^2 e^{- M_1 (h_{q+1} - \epsilon) + M_2 (h_q + \epsilon)}  \leq e^{- \gamma M_1}
\end{equation}
By property E, 
\begin{equation}
\label{eq4.12} 
\lambda \left( \bigcup_{B \in  \mathfrak{b}_{q+1}^{M_1}} B \right) \leq \frac{a}{M_1}.
\end{equation}
Putting \eqref{eq:estH}, \eqref{eq4.13}, and \eqref{eq4.12} together, we get 
$$
\lambda(H_2(q)) \leq  \frac{a}{M_1} + 2 c^2 e^{- M_1 (h_{q+1} - \epsilon) + M_2 (h_q + \epsilon)} = O \left( \frac{1}{M_1} + e^{-\gamma M_1} \right).  
$$
This completes the proof of Claim 4.

By Claim 4,  
\[\bigcap_{i=1}^q (T_i^{-m} A_i \cap T_i^{-n} A_i) = H_1(q) \cup H_2(q) \cup I_2(q), \]
where $H_1(q) \in  \sigma (\mathcal{A}_{q+1}^{M_1})$ and $\lambda (H_2(q) \cup I_2(q)) = O(\frac{1}{M_1} + e^{- \gamma M_1})$.

Note that $H_1(q) \cap T_{q+1}^{-m} A_{q+1}  \cap T_{q+1}^{-n} A_{q+1}  \in \sigma(\mathcal{A}_{q+1}^{n+l})$, thus one can write 
\begin{equation}
\label{eqn:$I_1(q+1)$}
H_1(q) \cap T_{q+1}^{-m} A_{q+1} \cap T_{q+1}^{-n} A_{q+1} = I_1(q+1) \cup R,
\end{equation}
where $I_1(q+1) = \bigcup_{i=1}^{s(q+1)}J_i(q+1)$ with $J_i(q+1) \in \mathfrak{g}_{q+1}^{n+l}$ for $1 \leq i \leq s(q+1)$ and $R$ is union of some intervals in $\mathfrak{b}_{q+1}^{n+l}$, so 
\begin{equation}
\label{eqn:R} 
\lambda (R) \leq  \frac{a}{n+l}  = O \left( \frac{1}{M_1} \right).
\end{equation}
Let $$I_2(q+1) = \bigcap_{i=1}^{q+1} (T_i^{-m} A_i \cap T_i^{-n} A_i)  \setminus I_1(q+1),$$ 
so $I_1 (q+1) \bigcup I_2(q+1) = \bigcap\limits_{i=1}^{q+1} (T_i^{-m} A_i \cap T_i^{-n} A_i).$ 
Note that
\[ I_2(q+1) \subset R \cup H_2(q) \cup I_2(q),\]
so $\lambda (I_2(q+1)) = O\left ( \frac{1}{M_1} + e^{-\gamma M_1} \right)$.

Now, utilizing the fact that $(T_{q+1}, \mu_{q+1}, \mathcal{A}_{q+1})$ has property M, we obtain 
\begin{align} 
\label{eqn:comp}
\lambda \left( H_1(q) \cap T_{q+1}^{-m} A_{q+1} \cap T_{q+1}^{-n} A_{q+1} \right) 
&= \lambda \left(H_1(q) \cap T_{q+1}^{-m} A_{q+1} \right) \mu_{q+1} (A_{q+1}) + O(b_{n-m-l})  \notag \\
&= \lambda(H_1(q)) \mu_{q+1} (A_{q+1})^2 + O(b_{m-M_1-l} + b_{n-m-l}).
\end{align}
(We use that $H_1(q) \cap T_{q+1}^{-m} A_{q+1} \in \mathcal{A}_{q+1}^{m+l}$ for the first equality and  $H_1(q) \in \mathcal{A}_{q+1}^{M_1} \subset \mathcal{A}_{q+1}^{M_1 + l}$ for the second equality.)

Then,
\begin{align*} 
\lambda(I_1(q+1)) &= \lambda \left(H_1(q) \cap T_{q+1}^{-m} A_{q+1} \cap T_{q+1}^{-n} A_{q+1} \right) 
+ O (\lambda(R))  \quad (\text{by } \eqref{eqn:$I_1(q+1)$}) \\
&= \lambda (H_1(q)) \mu_{q+1}^2(A_{q+1})^2 + O \left(b_{m-M_1-l} + b_{n-m-l} + \frac{1}{M_1}  \right) (\text{by } \eqref{eqn:R}, \eqref{eqn:comp})\\
&= \lambda (I_1(q)) \mu_{q+1}^2(A_{q+1})^2 + O \left(b_{m - M_1 -l} + b_{n-m-l} + \frac{1}{M_1} + e^{-\gamma M_1} \right) (\text{by Claim 4})
\end{align*}
\end{proof} 

The following result is a special case of Theorem \ref{j.e.:pointwise}, which was formulated in the Introduction.
\begin{Theorem}[Theorem \ref{Thm1.2:pje_TimesbGaussBeta}]
\label{Cor:pointwise}
Let $s, t \in \mathbb{N}$.
Let $b_1, \dots, b_s$ be distinct integers with $b_i \geq 2$ for all $i = 1, 2, \dots, s$ and let $\beta_1, \dots, \beta_t$ be distinct non-integers with $\beta_i > 1$ for all $i =1, 2, \dots, t$.
Then
\begin{enumerate}
\item $T_{b_1}, \dots, T_{b_s}, T_{\beta_1}, \dots T_{\beta_t}$ are pointwise jointly ergodic.
\item $T_G, T_{b_1}, \dots, T_{b_s}, T_{\beta_1}, \dots T_{\beta_t}$ are pointwise jointly ergodic if
$\log \beta_i \ne \frac{\pi^2}{6 \log 2}$ for $1 \leq i \leq t$. 
\end{enumerate} 
\end{Theorem}

\begin{Corollary}[Corollary \ref{cor:normality:basic}]
\label{Corollary4.4}
Let $s, t \in \mathbb{N}$. Let $b_1, \dots, b_s$ be distinct integers with $b_i \geq 2$ for all $i = 1, 2, \dots, s$ and let $\beta_1, \dots, \beta_t$ be distinct non-integers with $\beta_j > 1$ for all $j =1, 2, \dots, t$.
Then almost every $x \in [0,1)$ is jointly $(b_1, \dots, b_s, \beta_1, \dots, \beta_t)$-normal and also, if $\log \beta_j \ne \frac{\pi^2}{6 \log 2}$ for any $j= 1, 2 \dots, t$, almost every $x \in [0,1)$ is jointly $(b_1, \dots, b_s, \beta_1, \dots, \beta_t, CF)$-normal. 
\end{Corollary}

\section{Disjoint systems and pointwise joint ergodicity}
\label{sec:disjoint}  

In this section, we deal with an approach to pointwise joint ergodicity which is based on the notion of disjointness of dynamical systems that was introduced by Furstenberg in \cite{Fur}. We remark that while in the previous sections we were mostly dealing with the maps of the interval, the main results in this section are quite general. 
The reader is referred to the classical literature (\cite{Furbook}, \cite{Pe}, \cite{Rud}, \cite{Walt}) for definitions of mixing notions which are mentioned in this section. 

\begin{Definition}
Probability measure preserving systems $(X_i, \mathcal{B}_i, \mu_i, T_i), 1 \leq i \leq k,$ are {\em mutually disjoint} if the only $T_1 \times \cdots \times T_k$-invariant measure on $X_1 \times \cdots \times X_k$, which has the property that the projection of $\mu$ onto $X_i$ is $\mu_i$ for all $i = 1, \dots, k$, is the product measure $\mu_1 \times \cdots \times \mu_k$.
\end{Definition}
By slight abuse of language, we will say that $T_1, \dots, T_k$
are mutually disjoint if it is clear from the context what the systems $(X_i,\mathcal{B}_i, \mu_i)$ are. 

Here are some examples of disjoint measure preserving systems:
\begin{enumerate}[(1)]
    \item Translations on compact abelian groups and weakly mixing systems are disjoint (Theorem 6.10 in \cite{Rud}).
    \item Zero entropy systems and K-automorphisms are disjoint (Theorem 6.11 in \cite{Rud}).
    \item Rigid systems and mildly mixing systems are disjoint (Proposition 4.3 in \cite{KaTh}).
     \item If $T_1$ is a translation on a compact abelian group, $T_2$ is a weakly mixing rigid transformation, and $T_3$ is a mildly mixing transformation, then it is not hard to show that $T_1, T_2, T_3$ are mutually disjoint.  
     \item If $ \mathsf{X}_i = (X, \mathcal{B}, \mu_i, T_i), 1 \leq i \leq k,$ are ergodic, measurably distal systems and Kronecker factors of $\mathsf{X}_i$ are mutually disjoint, then $T_1, \dots, T_k$ are mutually disjoint (Proposition 8.5 in \cite{BeSo}). 
\end{enumerate}

The following simple theorem establishes a relation between disjointness and pointwise joint ergodicity.
\begin{Theorem}
\label{thm:disjoint:extension}
Let $X$ be a compact metric space and let $\mathcal{B}$ be the Borel $\sigma$-algebra of $X$. 
Let $\mu_1, \dots, \mu_l, \nu_1, \dots, \nu_m$ be  pairwise equivalent probability measures on $(X, \mathcal{B})$. 
Let $\mathsf{X}_i^{(1)} = (X, \mathcal{B}, \mu_i, S_i)$, $1 \leq i \leq l$, and $\mathsf{X}_j^{(2)} = (X, \mathcal{B}, \nu_j, T_j)$, $1 \leq j \leq m$, be measure preserving systems such that 
\begin{enumerate}[(1)]
\item $S_1, \dots, S_l$ are pointwise jointly ergodic with respect to $(\mu_1, \dots, \mu_l)$, 
\item $T_1, \dots, T_m$ are pointwise jointly ergodic with respect to $(\nu_1, \dots, \nu_m)$, 
\item $\mathsf{X}^{(1)} = \mathsf{X}_1^{(1)} \times \cdots \times \mathsf{X}_l^{(1)}$ and $\mathsf{X}^{(2)} = \mathsf{X}_1^{(2)} \times \cdots \times \mathsf{X}_m^{(2)}$ are disjoint.
\end{enumerate}
Then, $S_1, \dots, S_l, T_1, \dots, T_m$ are pointwise jointly ergodic with respect to $(\mu_1, \dots, \mu_l, \nu_1, \dots, \nu_m)$. 
\end{Theorem}

\begin{proof} 
By the assumptions (1) and (2), Theorem \ref{lem:pje_gp} implies that there exists a full measure set $A \in \mathcal{B}$ such that if $x \in A$, then
\[ \frac{1}{N} \sum_{n=0}^{N-1} \delta_{S_1^n x} \times \cdots \times \delta_{S_l^n x}   \rightarrow \mu_1 \times \cdots \times \mu_l \]
and 
\[ \frac{1}{N} \sum_{n=0}^{N-1} \delta_{T_1^n x} \times \cdots \times \delta_{T_m^n x}   \rightarrow \nu_1 \times \cdots \times \nu_l. \]
Now consider the sequence of measures on $X^{l+m}$ 
\[ \rho_{N} 
:= \frac{1}{N} \sum_{n=0}^{N-1} \delta_{S_1^n x} \times \cdots \times \delta_{S_l^n x} \times \delta_{T_1^n x} \times \cdots \times \delta_{T_m^n x}, \quad N= 1, 2, \dots. \]
Let $\rho$ be a weak* limit point of $(\rho_N)$ and let us show that $\rho = \mu_1 \times \cdots \times \mu_l \times \nu_1 \times \cdots \times \nu_m$. 
Note that $\rho$ is $(S_1 \times \cdots \times S_l) \times (T_1 \times \cdots \times T_m)$-invariant and that the projections of $\rho$ on $\mathsf{X}^{(1)}$ and $\mathsf{X}^{(2)}$ are, correspondingly, $\mu_1 \times \cdots \times \mu_l$ and $\nu_1 \times \cdots \times \nu_m$.
So, the desired result follows from the assumption that $\mathsf{X}^{(1)}$ and $\mathsf{X}^{(2)}$ are disjoint. 
\end{proof}

The following corollary of Theorem \ref{thm:disjoint:extension} provides a convenient criterion of pointwise joint ergodicity.
A special case when $k=2$ and $\mu_1 = \mu_2$ was obtained by Berend (see Theorem 2.2 in \cite{Be}; cf. also Theorem I.6 in \cite{Fur}).
\begin{Corollary}
\label{cor5.3}
Let $X$ be a compact metric space and let $\mathcal{B}$ be the Borel $\sigma$-algebra of $X$.
Let $\mathsf{X}_i = (X, \mathcal{B}, \mu_i, T_i)$, $1 \leq i \leq k$, be ergodic probability measure preserving systems. 
Suppose that  $\mu_1, \dots, \mu_k$ are pairwise equivalent and $\mathsf{X}_1, \dots, \mathsf{X}_k$ are mutually disjoint. Then $T_1, \cdots, T_k$ are pointwise jointly ergodic with respect to $(\mu_1, \dots, \mu_k)$.
\end{Corollary} 
\begin{proof} 
Take first $k=2$. Since $\mathsf{X}_1$ and $\mathsf{X}_2$ are disjoint, by Theorem \ref{thm:disjoint:extension}, $T_1, T_2$ are pointwise jointly ergodic. 
Take now $k=3$. 
Since $\mathsf{X}_1,  \mathsf{X}_2$ and $\mathsf{X}_3$ are mutually disjoint,  $\mathsf{X}_1 \times  \mathsf{X}_2$ and $\mathsf{X}_3$ are disjoint and so  by Theorem \ref{thm:disjoint:extension}, $T_1, T_2, T_3$ are pointwise jointly ergodic. And so on.
\end{proof}

\begin{Theorem}
\label{thm:dis}
Let $X$ be a compact metric space and let $\mathcal{B}$ be the Borel $\sigma$-algebra of $X$. 
Let $\mu_1, \dots, \mu_l, \nu_1, \dots, \nu_m$ be  pairwise equivalent probability measures on $(X, \mathcal{B})$. 
Let $\mathsf{X}_i^{(1)} = (X, \mathcal{B}, \mu_i, S_i)$, $1 \leq i \leq l$, and $\mathsf{X}_j^{(2)} = (X, \mathcal{B}, \nu_j, T_j)$, $1 \leq j \leq m$, be measure preserving systems such that 
\begin{enumerate}
    \item each $S_i$, $1 \leq i \leq l$, has zero entropy,
    \item each $T_i$, $1 \leq i \leq m$, is exact.
\end{enumerate}
If $ S_1, \dots, S_l$ are pointwise jointly ergodic with respect to $(\mu_1, \dots, \mu_l)$ and $T_1, \dots, T_m$ are pointwise jointly ergodic with respect to $(\nu_1, \dots, \nu_m)$, then $S_1, \dots, S_l, T_1, \dots, T_m$ are pointwise jointly ergodic with respect to $(\mu_1, \dots, \mu_l, \nu_1, \dots, \nu_m)$.
\end{Theorem}

\begin{proof}
Since a Cartesian product of zero entropy systems has zero entropy, $S_1 \times \cdots \times S_l$ has zero entropy. 
It is known  (\cite{Rok}) that natural extensions of exact transformations are K-automorphisms and a Cartesian product of K-automorphisms is K-automorphism. Thus, $T_1 \times \cdots \times T_m$ is a factor of a K-automorphism.  
It is also known that K-automorphisms are disjoint from zero entropy systems. Also, if $\mathsf{X}$ and $\mathsf{Y}$ are disjoint and $\mathsf{Z}$ is a factor of $\mathsf{X}$, then $\mathsf{Z}$ and $\mathsf{Y}$ are disjoint (Proposition I.1 in \cite{Fur}).  Thus, $S_1 \times \cdots \times S_l$ and $T_1 \times \cdots \times T_m$ are disjoint, and so Theorem \ref{disjoint-main-thm-pje} follows from \ref{thm:disjoint:extension}.  
\end{proof}

Theorems \ref{j.e.:pointwise} and \ref{thm:dis} lead to the following result.
\begin{Theorem}
\label{disjoint-main-thm-pje}
Let $\mu_1, \dots, \mu_k$ be probability measures on $([0,1], \mathcal{B})$, where $\mathcal{B}$ is the Borel $\sigma$-algebra of $[0,1]$, such that each $\mu_i$ is equivalent to the Lebesgue measure $\lambda$. 
For $1 \leq i \leq k$, let $T_i : [0,1] \rightarrow [0,1]$ be a $\mu_i$-preserving ergodic transformation. 
Let $m \leq k$.  In addition, assume that
\begin{enumerate}
\item for $1 \leq i \leq m$, $T_i$ has zero entropy and $T_1, \dots, T_m$ are pointwise jointly ergodic,
\item for $m+1 \leq i \leq k$, $(T_i, \mu_i, \mathcal{A}_i) \in \mathcal{T}$ along with an appropriate partition $\mathcal{A}_i$ such that $(T_i, \mu_i, \mathcal{A}_i)$ has properties E and M and entropy values $h(T_{m+1}), \cdots, h(T_k)$ are distinct.
\end{enumerate}
Then for any bounded measurable functions $f_1, \dots, f_k$,
\begin{equation}
\label{eq:Thm_Intro_PJE} 
\lim_{N \rightarrow \infty} \frac{1}{N} \sum_{n=0}^{N-1} f_1 ({T_1^n x})   f_2 ({T_2^n x}) \cdots  f_k ({T_k^n x}) = \prod_{i=1}^k \int f_i \, d \mu_i 
\end{equation}
for almost every $x \in [0,1]$.  
\end{Theorem}

The following extension of Theorem \ref{Thm1.2:pje_TimesbGaussBeta} (Theorem \ref{Cor:pointwise}) is an immediate corollary of Theorem \ref{disjoint-main-thm-pje}.

\begin{Corollary}
\label{ex:notThm4.1}
Let $\mu_0$ be a probability measure on $[0,1]$, which is equivalent to $\lambda$. Let $S: [0,1] \rightarrow [0,1]$ be a $\mu_0$-preserving, ergodic transformation with $h(S) = 0$.  
Let $s, t \in \mathbb{N}$. Let $b_1, \dots, b_s$ be distinct integers with $b_i \geq 2$ for all $i = 1, 2, \dots, s$ and let $\beta_1, \dots, \beta_t$ be distinct non-integers with $\beta_i > 1$ for all $i =1, 2, \dots, t$. Then
\begin{enumerate}
\item $S, T_{b_1}, \dots, T_{b_s}, T_{\beta_1}, \dots T_{\beta_t}$ are pointwise jointly ergodic:\\
 for any bounded measurable functions $f_0, f_1, \dots, f_s, g_1, \dots, g_t$,
\begin{equation*}
\lim_{N \rightarrow \infty} \frac{1}{N} \sum_{n=0}^{N-1} f_0(S^n x)  \prod_{i=1}^s f_i ({b_i^n x}) \prod_{j=1}^t g_j ({T_{\beta_j}^n x})  = \int f_0 \, d \mu_0 \, \prod_{i=1}^s \int f_i \, d \lambda  \, \prod_{j=1}^t \int g_j \, d \mu_{\beta_j}
\end{equation*}
for almost every $x \in [0,1]$.  
\item $S, T_G, T_{b_1}, \dots, T_{b_s}, T_{\beta_1}, \dots T_{\beta_t}$ are pointwise jointly ergodic if
$\log \beta_i \ne \frac{\pi^2}{6 \log 2}$ for $1 \leq i \leq t$: \\
for any bounded measurable functions $f_0,  g_0, f_1, \dots, f_s, g_1, \dots, g_t$,
\begin{equation*} 
\lim_{N \rightarrow \infty} \frac{1}{N} \sum_{n=0}^{N-1} f_0 (S^n x) \, g_0 (T_G^n x) \, \prod_{i=1}^s f_i ({b_i^n x}) \, \prod_{j=1}^t g_j ({T_{\beta_j}^n x})  = \int f_0 \, d \mu_0 \int g_0 \, d \mu_G  \, \prod_{i=1}^s \int f_i \, d \lambda \, \prod_{j=1}^t \int g_j \, d \mu_{\beta_j}
\end{equation*}
for almost every $x \in [0,1]$. 
\end{enumerate} 
\end{Corollary}

\section{Equivalent forms of normality and joint normality of F-representations}
\label{Sec:FRep}

In this section we deal with various equivalent forms of normality and joint normality, which were briefly discussed in the Introduction.
We start with proving Theorem \ref{Intro:equiv:normal:F-representation} from the Introduction, which we formulate here again for reader's convenience. 

\begin{Theorem}[Theorem \ref{Intro:equiv:normal:F-representation}]
\label{equiv:normal:F-representation} 
Let ${\mathbf{T}} = (T, \mu, \mathcal{A}) \in \tilde{\mathcal{T}}$ such that $T$ is totally ergodic. Let $x \in [0,1]$.
The following statements are equivalent.
\begin{enumerate}[(i)]
\item $x$ is ${\mathbf{T}}$-normal. 
\item $(T^n x)$ is equidistributed in $[0,1]$ with respect to $\mu$, that is, for any continuous function $f:[0,1] \rightarrow \mathbb{R}$,
\begin{equation}
\label{condii:thm6.1} 
\lim_{N \rightarrow \infty} \frac{1}{N} \sum_{n=0}^{N-1} f(T^n x) = \int f \, d \mu.
\end{equation}
\item For any $m \geq 2$, $x$ is ${\mathbf{T}}_m$-normal. 
\item $T^k x$ is simply ${\mathbf{T}}_m$-normal for all $k \geq 0$ and $m \geq 1$.
\item $x$ is simply ${\mathbf{T}}_m$-normal for all $m \geq 1$. 
\item There exists a sequence $(m_l)_{l \in \mathbb{N}}$ with $m_1 < m_2 < \cdots$ such that $x$ is simply ${\mathbf{T}}_{m_l}$-normal for any $l \in \mathbb{N}$.
\end{enumerate}
\end{Theorem}

The proof of Theorem \ref{equiv:normal:F-representation} rests on two useful lemmas which are of independent interest. Before formulating and proving these lemmas we need to introduce some additional terminology. 

Let $X$ be a compact metric space and let $M(X)$ denote the space of probability Borel measures on $X$. 
For $\mu \in M(X)$, let $\mathcal{R}(X, \mu)$ denote the space of Riemann integrable functions on $X$. Note that  $f \in \mathcal{R} (X, \mu)$ if and only if $\mu(D_f) = 0$, where $D_f$ is the set of points of discontinuity of $f$. 
Also, note that if  the sequence $\mu_1, \mu_2, \dots \in M(X)$ converges to the measure $\mu$ in weak* topology, then, by the standard approximation argument, for any $f \in \mathcal{R} (X, \mu)$,
 \[\int f \, d \mu_n \rightarrow \int f \, d \mu.\]

We will say that a measure preserving system $(X, \mathcal{B}, \mu, T)$ is {\em admissible} if $X$ is a compact metric space, $\mathcal{B}$ is  the $\sigma$-algebra of Borel sets of $X$, $\mu$ is a probability measure on $\mathcal{B}$, and $T: X \rightarrow X$ is a $\mu$-preserving transformation with the property that the set $S$ of points of discontinuity of $T$ satisfies $\mu(S) = 0$. 
Following \cite{GHL}, we will call any such transformation {\em $\mu$-continuous}.
Note that if  $(X, \mathcal{B}, \mu, T)$ is an admissible system, so is $(X, \mathcal{B}, \mu, T^m)$ for any $m \geq 2$.

Given an admissible system $(X, \mathcal{B}, \mu, T)$ and a point $x \in X$ we will say that $x$ is  {\em $(T, \mu)$-generic} if for any continuous function $f$ on $X$, 
\[ \lim_{N \rightarrow \infty} \frac{1}{N} \sum_{n=0}^{N-1} f(T^n x) = \int f \, d \mu,\]
or, equivalently,
\[\frac{1}{N} \sum_{n=0}^{N-1} \delta_{T^n x} \rightarrow \mu \,\, \text{ in weak* topology.} \]
Note that a point $x \in X$ is $(T, \mu)$-generic if and only if  for any $f \in \mathcal{R} (X, \mu)$, 
\[ \lim_{N \rightarrow \infty} \frac{1}{N} \sum_{n=0}^{N-1} f(T^n x) = \int f \, d \mu.\]

\begin{Remark}
\label{Rem:genequiv}
Let $\mathbf{T} = (T, \mu, \mathcal{A}) \in \tilde{\mathcal{T}}$. Then $([0,1], \mathcal{B}, \mu, T)$, where $\mathcal{B}$ is the Borel $\sigma$-algebra of $[0,1]$, is an admissible system.
Moreover, the following statements are equivalent.
\begin{enumerate}[(i)]
\item $x \in [0,1]$ is $(T, \mu)$-generic (i.e. for any $f \in C ([0,1])$, $\lim\limits_{N \rightarrow \infty} \frac{1}{N} \sum\limits_{n=0}^{N-1} f(T^n x) = \int f \, d \mu$).
\item  For any interval $I \subset [0,1]$, 
\begin{equation}
\label{eq:rem:genequiv}
   \lim_{N \rightarrow \infty} \frac{\# \{ 0 \leq n \leq N - 1: T^n x  \in  I \}}{N} =\mu(I). 
\end{equation}
\end{enumerate}
(One has only to justify (ii) $\Rightarrow$ (i). It is not hard to see that this implication follows form the fact that linear combinations of characteristic functions of intervals are dense in $\mathcal{R} ([0,1], \mu)$.)
\end{Remark}

The following lemma is a natural dynamical version of a classical result due to Pyatetskii-Shapiro. (See Theorem 2 in \cite{Py}. See also Theorem 3 in \cite{PoPy} and  Lemma 8.1 in Chapter 1 of \cite{KN}.)
\begin{Lemma}
\label{criteria-generic}
Let  $(X, \mathcal{B}, \mu, T)$ be an ergodic admissible system. 
If $x \in X$ has the property that there is a constant $C > 0$ such that for any non-negative continuous function $f$ on $X$ 
\begin{equation}
\label{cond:equiv} 
\limsup_{N \rightarrow \infty} \frac{1}{N} \sum_{n=0}^{N-1} f(T^n x) \leq C  \int f \, d \mu,
\end{equation}
then $x$ is $(T, \mu)$-generic.
\end{Lemma}

\begin{proof} 
Define a sequence of measures in $M(X)$
\[ \nu_N := \frac{1}{N} \sum_{n=0}^{N-1} \delta_{T^n x}, \quad N \in \mathbb{N}. \]
It is sufficient to show that $\nu_N \rightarrow \mu$ in weak* topology.

Let $\nu$ be a weak* limit point of the sequence $(\nu_N)$ and let  $(\nu_{N_k})$ be a subsequence of $(\nu_N)$  such that $\nu_{N_k} \rightarrow \nu$.  
Note that by \eqref{cond:equiv}, for any nonnegative continuous function $f$ on $X$,  
\[ \int f d \nu = \lim_{k \rightarrow \infty} \int f \, d \nu_{N_k} = \lim_{k \rightarrow \infty}  \frac{1}{N_k} \sum_{n=0}^{N_k-1} f(T^n x)   \leq C \int f d \mu,\]
so $\nu \ll \mu$ and hence $\mathcal{R} (X, \mu) \subset \mathcal{R} (X, \nu)$.

Now let us show that $\int f \, d \nu = \int f \circ T \, d \nu$ for any $f \in C(X)$.   
Note that since $T$ is $\mu$-continuous, for any $f \in C(X)$, $f \circ T \in \mathcal{R} (X, \mu) \subset \mathcal{R} (X, \nu)$, and so 
\[\ \int f \, d \nu_{N_k} \rightarrow \int f \, d \nu \quad \text{and} \quad \int f \circ T \, d \nu_{N_k} \rightarrow \int f \circ T \, d \nu.\]
Since \[ \left| \int f \, d \nu_{N_k} -  \int f \circ T \, d \nu_{N_k} \right| 
= \left| \frac{1}{N_k} \sum_{n=0}^{N_k -1 } \left( f(T^n x) - f(T^{n+1} x) \right) \right| 
\leq \frac{2}{N_k} \| f \|_u \rightarrow 0 \quad \text{as } k \rightarrow \infty,\]
we get $\int f \, d \nu = \int f \circ T \, d \nu$ for any $f \in C(X)$.  
This, in turn, implies that  for any  $f \in C(X)$
\begin{equation}
\label{eq6.3}
 \int f \circ T^n d \nu = \int f d \nu \,  \text{ for any } n \in \mathbb{N}. 
 \end{equation}
Fix a continuous function $f$ on $X$. Since $T$ is ergodic, the pointwise ergodic theorem implies 
\begin{equation} 
\label{eq6.4}
\lim_{N \rightarrow \infty} \frac{1}{N} \sum_{n=0}^{N-1} f(T^n y) = \int f \, d \mu
\end{equation}
for $\mu-\text{a.e. } y$. Moreover, the fact that $\nu \ll \mu$ implies that \eqref{eq6.4} also holds for $\nu-\text{a.e. } y$. 
Note that
\begin{equation}
\label{eq6.5} 
\left| \frac{1}{N} \sum_{n=0}^{N-1} f(T^n y) \right| \leq \| f \|_u \quad \text{for any } y \in X.
\end{equation}
Using  \eqref{eq6.3} and \eqref{eq6.5} and the fact that \eqref{eq6.4} holds for $\nu-\text{a.e. } y$, we get with the help of the dominated convergence  theorem
\begin{align*}
\int f \, d \nu &=  \int \frac{1}{N} \sum_{n=0}^{N-1} f(T^n y) d \nu(y)  = \lim_{N \rightarrow \infty} \int \frac{1}{N} \sum_{n=0}^{N-1} f(T^n y) d \nu(y) =  \int \lim_{N \rightarrow \infty}  \frac{1}{N} \sum_{n=0}^{N-1} f(T^n y) d \nu(y) \\
&= \int \int f \, d \mu \, d \nu = \int f d \mu.
\end{align*}
Thus $\nu = \mu$, which implies that  any limit point of $(\nu_N)$ is $\mu$, and so $\nu_N \rightarrow \mu$ in weak* topology.
\end{proof}

\begin{Lemma}
\label{Sec6.1:Thm6.3}
Let  $(X, \mathcal{B}, \mu, T)$ be a totally ergodic admissible system.
Let $m \geq 2$.
Then $x \in X$ is $(T, \mu)$-generic if and only if $x$ is $(T^m, \mu)$-generic.
\end{Lemma}

\begin{proof}
Suppose that $x$ is $(T, \mu)$-generic.
Note that for any non-negative continuous function $f$ on $X$, 
\[  \frac{1}{N} \sum_{n=0}^{N-1} f(T^{mn} x) \leq m  \frac{1}{mN} \sum_{n=0}^{mN-1} f(T^{n} x),\]
and hence
\[ \limsup_{N \rightarrow \infty} \frac{1}{N} \sum_{n=0}^{N-1} f(T^{mn} x) \leq m \int f \, d \mu.\]
By Lemma \ref{criteria-generic}, $x$ is $(T^m, \mu)$-generic.

Now suppose that $x$ is $(T^m, \mu)$-generic.
Let $f$ be a non-negative continuous function on $X$. Then for any $N \in \mathbb{N}$ we have
\[ \frac{1}{N} \sum_{n=0}^{N-1}  f(T^n x) 
\leq \frac{[N/m] +1}{N} \sum_{r=0}^{m-1} \frac{1}{[N/m] +1} \sum_{n=0}^{[N/m]}  f \circ T^r (T^{mn} x). \]
Since  $T$ is $\mu$-continuous,  $f \circ T^r$ is Riemann integrable  for each $r = 0, 1, \dots, m-1$.
Therefore
\[ \lim_{N \rightarrow \infty} \frac{1}{[N/m]+1} \sum_{n=0}^{[N/m]}  f \circ T^r (T^{mn} x) =  \int f \circ T^r \, d \mu = \int f \, d \mu.\]
Hence
\[  \limsup_{N \rightarrow \infty} \frac{1}{N} \sum_{n=0}^{N-1}  f(T^n x) \leq  \int f \, d \mu,\]
which, by Lemma \ref{criteria-generic}, implies  $x$ is $(T, \mu)$-generic. 
\end{proof}

\begin{proof}[Proof of Theorem \ref{equiv:normal:F-representation}]

For $\mathbf{T} = (T, \mu, \mathcal{A}) \in \tilde{\mathcal{T}}$, write $\mathcal{A} = \{ I_j:  j \in D\}$, where $D$ is a finite or countably infinite index set. For a string $S = (a_1, \dots, a_k) \in D^k$, let $C_{\mathbf{T}} (S)$ be the cylinder set in $\mathcal{A}^k$ corresponding to $S$. Let $1_{C_\mathbf{T}(S)}$ denote the characteristic function of $C_{\mathbf{T}} (S)$.
For $x \in [0,1]$, let $(r_n(x))$ be the F-representation of $x$.  Note that for any $k \in \mathbb{N}$ and $S = (a_1, \dots, a_k) \in D^k$, 
\begin{equation}
\label{eq1:pf:thm6.1}
    \sum_{n=0}^{N-1} 1_{C_\mathbf{T}(S)}(T^n x) = \# \{ 1 \leq n \leq N: r_n(x) = a_1, \dots, r_{n+k-1}(x) = a_k\}.
\end{equation} 

(i) $\Rightarrow$ (ii): Let $x \in [0,1]$ be $\mathbf{T}$-normal. 
Then \eqref{eq1:pf:thm6.1} implies that for any $k \in \mathbb{N}$ and $S \in D^k$ 
\begin{equation}
\label{eq*:pf}
    \lim_{N \rightarrow \infty} \frac{1}{N} \sum_{n=0}^{N-1} 1_{C_\mathbf{T}(S)} (T^n x) = \int 1_{C_\mathbf{T}(S)} \, d \mu.
\end{equation} 
By linearity and the monotone convergence theorem, \eqref{eq*:pf} implies that  for any $ m \in \mathbb{N}$ and $I \in \sigma(\mathcal{A}^m)$,
\begin{equation}
\label{eq*2:pf}
    \lim_{N \rightarrow \infty} \frac{1}{N} \sum_{n=0}^{N-1} 1_I (T^n x) = \int 1_I \, d \mu.
\end{equation}
Since the partition $\mathcal{A}$ generates the $\sigma$-algebra $\mathcal{B}$, the standard approximation argument implies that \eqref{eq*2:pf} holds for any interval in $[0,1]$, and so, by Remark \ref{Rem:genequiv}, $x$ is $(T, \mu)$-generic.

(ii) $\Rightarrow$ (i): If $x \in [0,1]$ is $(T, \mu)$-generic, then \eqref{eq1:pf:thm6.1} and \eqref{eq:rem:genequiv} imply that for any $k \in \mathbb{N}$ and $S = (a_1, \dots, a_k) \in D^k$
\[ \lim_{N \rightarrow \infty} \frac{\# \{ 1 \leq n \leq N: r_n(x) = a_1, \dots, r_{n+k-1}(x) = a_k\}}{N}  = \lim_{N \rightarrow \infty} \frac{1}{N} \sum_{n=0}^{N-1} 1_{C_\mathbf{T}(S)}(T^n x) = \mu (C_\mathbf{T}(S)), \]
so $x$ is $\mathbf{T}$-normal.

(i) $\Leftrightarrow$ (iii) follows from Lemma \ref{Sec6.1:Thm6.3}. 

(i) $\Rightarrow$ (iv): If $x$ is $\mathbf{T}$-normal, then clearly $T^k x$ is $\mathbf{T}$-normal for any $k \geq 0$. Then since (i) is equivalent to (iii), $T^k x$ is $\mathbf{T}_m$-normal, which gives (iv). 

It is clear that (iv) $\Rightarrow$ (v) $\Rightarrow$ (vi). 

We will conclude the proof by showing (vi) $\Rightarrow$ (i). 
Let us assume that the index set $D$ is countably infinite (the proof for the case when $D$ is finite is analogous and is omitted).
By \eqref{eq*:pf}, it is enough to show that for any $k \in \mathbb{N}$ and $S = (a_1, \dots, a_k) \in D^k$
\begin{equation}
\label{eq:normal:cond:n}
\lim_{N \rightarrow \infty} \frac{1}{N} \sum_{n=0}^{N-1} 1_{C_\mathbf{T}(S)} (T^n x) = \mu(C_\mathbf{T}(S)).
\end{equation}
Let $m_l$ be a natural number with $m_l > k$. 
Let $(A_i)_{i=1}^{\infty}$ be an enumeration of strings in $D^{m_l}$.   
Since $C_\mathbf{T} (A_i) \cap C_\mathbf{T}(A_j) = \emptyset$ for $i \ne j$ and $[0,1] = \bigcup\limits_{i=1}^{\infty} C_\mathbf{T} (A_i) \, (\bmod \, \lambda)$, we have that $\sum\limits_{i=1}^{\infty} \mu(C_\mathbf{T} (A_i)) = 1.$
Let $\epsilon > 0$. 
Take $M \in \mathbb{N}$ such that $\sum\limits_{i = M+1}^{\infty} \mu(C_\mathbf{T} (A_i)) < \epsilon$. 

For $i \in \mathbb{N}$, let $k_i \in \mathbb{N}$ denote the number of occurrences of the string $S$ in the string $A_i$. 
Note that $0 \leq k_i \leq m_l - k +1$ for all $i = 1, 2, \dots, $ and
\begin{equation}
\label{est1:thm6.5:n} 
\sum_{n=0}^{N-1} 1_{C_\mathbf{T}(S)} (T^n x )  = \sum_{0 \leq n < \frac{N}{m_l}} \sum_{i=1}^{\infty} k_i 1_{C_\mathbf{T}(A_i)} (T^{m_l n} x) + O \left( \frac{N}{m_l} k \right). 
\end{equation}
We have
\begin{equation}
\label{est2:thm6.5:n} 
 \sum_{0 \leq n < \frac{N}{m_l}} \sum_{i=1}^{\infty} k_i 1_{C_{\mathbf{T}}(A_i)} (T^{m_l n} x) 
= \sum_{0 \leq n < \frac{N}{m_l}} \sum_{i=1}^{M} k_i 1_{C_{\mathbf{T}}(A_i)} (T^{m_l n} x) +
 \sum_{0 \leq n < \frac{N}{m_l}} \sum_{i=M+1}^{\infty} k_i 1_{C_{\mathbf{T}}(A_i)} (T^{m_l n} x).
 \end{equation}
Since $x$ is simply ${\mathbf{T}}^{m_l}$-normal, 
\begin{equation}
\label{est3:thm6.5:n} 
\lim_{N \rightarrow \infty} \frac{1}{N} \sum_{0 \leq n < \frac{N}{m_l}} \sum_{i=1}^{M} k_i 1_{C_{\mathbf{T}}(A_i)} (T^{m_l n} x) = \frac{1}{m_l} \sum_{i=1}^M k_i \mu({C_{\mathbf{T}}(A_i)}).
\end{equation}
Since for any $i \in \mathbb{N}$, $k_i \leq m_l$, we get the following estimate of the second term on the right hand side of \eqref{est2:thm6.5:n}:
\[  \sum_{0 \leq n < \frac{N}{m_l}} \sum_{i=M+1}^{\infty} k_i 1_{C_{\mathbf{T}}(A_i)} (T^{m_l n} x) 
\leq  m_l  \cdot \# \{0 \leq n < \frac{N}{m_l}:  T^{m_l n} x \notin \bigcup_{i=1}^{M} C_{\mathbf{T}} (A_i)\}. \]
We have 
\begin{align*} 
\lim_{N \rightarrow \infty} \frac{m_l}{N} \cdot \# \{0 \leq n < \frac{N}{m_l}:  T^{m_l n} x \notin \bigcup_{i=1}^{M} C_{\mathbf{T}} (A_i)\} 
 &= 1 -  \lim_{N \rightarrow \infty} \frac{m_l}{N} \cdot \# \{0 \leq n < \frac{N}{m_l}:  T^{m_l n} x \in \bigcup_{i=1}^{M} C_{\mathbf{T}} (A_i)\}  \\
 &= 1 - \sum_{i=1}^M \mu (C_{\mathbf{T}} (A_i)) = \sum_{i=M+1}^{\infty} \mu (C_{\mathbf{T}} (A_i)) < \epsilon.
 \end{align*}
 So 
 \begin{equation}
 \label{eq:6.5:n}
 \limsup_{N \rightarrow \infty} \frac{1}{N}  \sum_{0 \leq n < \frac{N}{m_l}} \sum_{i=M+1}^{\infty} k_i 1_{C_{\mathbf{T}}(A_i)} (T^{m_l n} x) \leq \epsilon.
 \end{equation}
Thus, by \eqref{est1:thm6.5:n}, \eqref{est2:thm6.5:n}, \eqref{est3:thm6.5:n} and \eqref{eq:6.5:n},
\begin{equation}
\label{est:Thm6.5:eq6.3:n}
 \limsup_{N \rightarrow \infty}  \left| \frac{1}{N} \sum_{n=0}^{N-1} 1_{C_\mathbf{T}(S)} (T^n x) - \frac{1}{m_l} \sum_{i=1}^M k_i \mu (C_{\mathbf{T}}(A_i)) \right|  = O \left(\epsilon + \frac{k}{m_l} \right).
 \end{equation}
Invoking the definition of $k_i$, we have for almost every $y \in [0,1)$
\begin{equation}
    \label{eq:new6.11:n}
 \sum_{i=1}^{\infty} k_i 1_{C_{\mathbf{T}}(A_i)} (y) =   \sum_{j=0}^{m_l - k} 1_{C_{\mathbf{T}}(S)} \circ T^j(y), 
 \end{equation}
and hence 
 \begin{equation}
    \label{eq:new6.11:2:n}
 0 \leq \sum_{j=0}^{m_l - k} 1_{C_{\mathbf{T}}(S)} \circ T^j(y) - \sum_{i=1}^{M} k_i 1_{C_{\mathbf{T}}(A_i)} (y) =  \sum_{i=M+1}^{\infty} k_i 1_{C_{\mathbf{T}}(A_i)} (y) . 
 \end{equation}
Since $\int 1_{C_{\mathbf{T}}(S)} \circ T^j \, d \mu = \mu (C_{\mathbf{T}}(S))$ for $0 \leq j \leq m_l-k$ and $k_i \leq m_l$ for all $i \in \mathbb{N}$,  from \eqref{eq:new6.11:2:n} we get 
\begin{equation*}
\label{eq:new6.12:2:n}
0 \leq (m_l -k) \mu(C_{\mathbf{T}}(S)) - \sum\limits_{i=1}^M k_i \mu(C_{\mathbf{T}} (A_i)) \leq m_l\sum\limits_{i=M+1}^{\infty}  \mu(C_{\mathbf{T}} (A_i)) \leq m_l \epsilon,
\end{equation*}
and so
\begin{equation*}
\label{eq:new6.12:n}
0 \leq \left(1 - \frac{k}{m_l}  \right) \mu(C_{\mathbf{T}}(S)) - \frac{1}{m_l}\sum\limits_{i=1}^M k_i \mu(C_{\mathbf{T}} (A_i)) \leq \epsilon,
\end{equation*}
which in turn implies that 
\begin{equation}
\label{eq:new6.12:nn}
\left| \mu(C_{\mathbf{T}}(S)) - \frac{1}{m_l}\sum\limits_{i=1}^M k_i \mu(C_{\mathbf{T}} (A_i)) \right| 
= O \left(\epsilon + \frac{k}{m_l} \right).
\end{equation}
From \eqref{est:Thm6.5:eq6.3:n} and \eqref{eq:new6.12:nn} we get
\begin{equation}
\label{eq6.12:n} 
\limsup_{N \rightarrow \infty} \left| \frac{1}{N} \sum_{n=0}^{N-1} 1_{C_\mathbf{T}(S)} (T^n x)  - \mu (C_{\mathbf{T}}(S)) \right| = O \left(\epsilon + \frac{k}{m_l} \right).
\end{equation}
Since $\epsilon$ is arbitrary, \eqref{eq6.12:n} implies
\begin{equation}
\label{eq6.13:n}
 \limsup_{N \rightarrow \infty} \left| \frac{1}{N} \sum_{n=0}^{N-1} 1_{C_\mathbf{T}(S)} (T^n x)  - \mu (C_{\mathbf{T}}(S)) \right| =  O \left( \frac{k}{m_l} \right).
\end{equation}
Taking $m_l \rightarrow \infty$ in \eqref{eq6.13:n} gives us \eqref{eq:normal:cond:n}. This concludes the proof of Theorem \ref{equiv:normal:F-representation}.
\end{proof}

\begin{Remark}
Let $b \in \mathbb{N}, b \geq 2$ and let $x = \sum\limits_{n=1}^{\infty} \frac{x_n}{b^n}$ be a $b$-normal number. 
A classical theorem due to Wall \cite{Wall} states that for any $m \in \mathbb{N}$ and $k \in \mathbb{N} \cup \{0\}$, the number $y = \sum\limits_{n=1}^{\infty} \frac{x_{mn+k}}{b^n}$ is also $b$-normal. 
This result easily follows form Theorem  \ref{intro:equiv:b-normal} (which is a rather special corollary of Theorem \ref{equiv:normal:F-representation}). 
Indeed, Theorem \ref{intro:equiv:b-normal} implies that if $x$ is normal in base $b$, then $b^k x$ is normal in base $b^m$,
and so for any $t \in \mathbb{N}$ and any string $S = (v_1, \dots, v_t)$ in $\{0, 1, \dots, b-1\}^t$ we have
\begin{align*} 
& \{ 0 \leq n \leq N-1:  x_{mn+k} = v_1, x_{m(n+1) +k} = v_2, \dots, x_{m (n+t-1) +k} =v_t\} \\
&\quad = \{ 0 \leq n \leq N-1:  T_b^{mn} (T_b^k x) \in  [v_1] \cap T_b^{-m} [v_2] \cap \cdots \cap  T_b^{-m(t-1)}[v_t]\}. 
\end{align*}
Hence
\begin{align} 
\label{eq:wall}
& \lim_{N \rightarrow \infty}\frac{\# \{0 \leq n \leq N-1: x_{mn+k} = v_1, x_{m(n+1) +k} = v_2, \dots, x_{m (n+t-1) +k} =v_t \}}{N} \notag \\
 & \quad \quad = \lambda ([v_1] \cap T_b^{-m} [v_2] \cap \cdots \cap T_b^{-m(t-1)} [v_t]) \notag \\
 & \quad \quad = \lambda ([v_1, v_2, \dots, v_t])
\end{align} 

Wall's theorem leads to the question whether a similar phenomenon of normality-preservation along arithmetic progressions takes place for general $\mathbf{T}$-normal  sequences. 
This question has, in general, a negative answer. Indeed, while for any ${\mathbf{T}} = (T, \mu, \mathcal{A}) \in \tilde{\mathcal{T}}$, we have that $\mathbf{T}$-normality of $x$ implies the $\mathbf{T}_m$-normality of $T^k x$ for any $m \in \mathbb{N}$ and $k \in \mathbb{N} \cup \{0\}$, the last equality in \eqref{eq:wall} may not hold for ${\mathbf{T}} = (T, \mu, \mathcal{A}) \in \tilde{\mathcal{T}}$, so arithmetic progression may not preserve  ${\mathbf{T}}$-normality. 
In fact, it was shown (\cite{HeVa}) that CF-normality is not preserved along arithmetic progression. (See also \cite{AbDo} for more discussion on the failure of normality preservation.)
\end{Remark}

We conclude this section by discussing a version of Theorem \ref{equiv:normal:F-representation} for joint normality.
\begin{Theorem}[Theorem \ref{intro:equiv:jointnormal:F-representation}]
\label{equiv:jointnormal:F-representation} 
Let ${\mathbf{T}}^{(i)}= (T_i, \mu_i, \mathcal{A}_i) \in \tilde{\mathcal{T}}, \, i = 1, 2 \dots, k$ such that $T_1 \times \cdots \times T_k$ is totally ergodic with respect to $\mu_1 \times \cdots \times \mu_k$.  Let $x \in [0,1]$.
The following are equivalent.
\begin{enumerate}[(i)]
\item $x$ is jointly $({\mathbf{T}}^{(1)}, \dots, {\mathbf{T}}^{(k)})$-normal. 
\item $k$-tuple $(x, \dots, x)$ is  $(T_1 \times \cdots \times T_k, \mu_1 \times \cdots \times \mu_k)$-generic: $(T_1^n x, \dots, T_k^n x)$ is equidistributed in $[0,1]^k$ with respect to $\mu_1 \times \cdots \times \mu_k$. 
\item For any $m \geq 2$, $x$ is jointly $({\mathbf{T}} ^{(1)}_m, \dots, {\mathbf{T}}^{(k)}_m)$-normal. 
\item$x$ is jointly simply $({\mathbf{T}} ^{(1)}_m, \dots, {\mathbf{T}}^{(k)}_m)$-normal for all $m \in \mathbb{N}$. 
\item There exists a sequence $(m_l)_{l \in \mathbb{N}}$ with $m_1 < m_2 < \cdots$ such that $x$ is jointly simply $({\mathbf{T}} ^{(1)}_{m_l}, \dots, {\mathbf{T}}^{(k)}_{m_l})$-normal.
\end{enumerate}
\end{Theorem}
Similarly to the proof of Theorem \ref{equiv:normal:F-representation}, which is based on Lemma \ref{criteria-generic} and Lemma \ref{Sec6.1:Thm6.3}, 
Theorem \ref{equiv:jointnormal:F-representation} follows from two lemmas, which are natural extensions of Lemma \ref{criteria-generic} and Lemma \ref{Sec6.1:Thm6.3}.

Before formulating these lemmas, observe that 
 if  $(X, \mathcal{B}, \mu_i, T_i)$ is an admissible system for $1 \leq i \leq k$, then so is the product system $(X^k, \otimes_{1}^k \mathcal{B}, \mu_1 \times \cdots \times \mu_k, T_1 \times \cdots \times T_k)$. 
 So the following two lemmas are corollaries of, correspondingly, Lemma \ref{criteria-generic} and Lemma \ref{Sec6.1:Thm6.3} applied to the product system $(X^k, \otimes_{1}^k \mathcal{B}, \mu_1 \times \cdots \times \mu_k, T_1 \times \cdots \times T_k)$.

\begin{Lemma}
\label{criteria-generic-joint}
Let  $(X, \mathcal{B}, \mu_i, T_i)$ be an ergodic admissible system for $1 \leq i \leq k$. 
Assume that  $T_1 \times \cdots \times T_k$ is ergodic with respect to $\mu_1 \times \cdots \times \mu_k$.
 
 If $x \in X$ has the property that there is a constant $C > 0$ such that for any non-negative continuous functions $f_1, \dots, f_k$,  
\[ \limsup_{N \rightarrow \infty} \frac{1}{N} \sum_{n=0}^{N-1} \prod_{i=1}^k f_i(T_i^n x) \leq C \prod_{i=1}^k \int f_i \, d \mu_i, \]
then the $k$-tuple $(x, \dots, x)$ is $(T_1 \times \cdots \times T_k, \mu_1 \times \cdots \times \mu_k)$-generic.
\end{Lemma}

\begin{Lemma}
\label{lem:6.11}
Let  $(X, \mathcal{B}, \mu_i, T_i)$ be an ergodic admissible system for $1 \leq i \leq k$ and assume that $T_1 \times \cdots \times T_k$ is totally ergodic with respect to $\mu_1 \times \cdots \times \mu_k$.
Let $x \in X$ and let $m \geq 2$.
Then the $k$-tuple $(x, \dots, x)$ is  $(T_1 \times \cdots \times T_k, \mu_1 \times \cdots \times \mu_k)$-generic if and only if the $k$-tuple $(x, \dots, x)$ is  $(T_1^m \times \cdots \times T_k^m, \mu_1 \times \cdots \times \mu_k)$-generic 
\end{Lemma}

The proof of Theorem \ref{equiv:jointnormal:F-representation} follows now from Lemmas \ref{criteria-generic-joint} and \ref{lem:6.11} and is totally analogous to the proof of  Theorem \ref{equiv:normal:F-representation} above.


\end{document}